\documentclass[11pt]{amsart}
\usepackage[T1]{fontenc}
\usepackage[utf8]{inputenc}
\usepackage{lmodern}
\usepackage{amsmath,amssymb,amsthm,mathtools}
\usepackage{graphicx,booktabs,array,microtype}
\usepackage{xcolor}
\usepackage{hyperref}
\hypersetup{hidelinks,hypertexnames=false,
 pdftitle={Sharp bounds, stability, and arithmetic fluctuations for induced pentagons},
 pdfauthor={Youwen Xu},pdfsubject={Induced pentagons under an edge constraint; exact certificates and stability},
 pdfkeywords={induced pentagons, edge inducibility, Clebsch graph, stability}}
\newtheorem{theorem}{Theorem}[section]
\newtheorem{lemma}[theorem]{Lemma}
\newtheorem{proposition}[theorem]{Proposition}
\newtheorem{corollary}[theorem]{Corollary}
\theoremstyle{definition}

\theoremstyle{remark}
\newtheorem{remark}[theorem]{Remark}

\title[Induced pentagons with an edge constraint]{Sharp bounds, stability, and arithmetic fluctuations for induced pentagons}

\author[]{Youwen Xu}
\hypersetup{pdfauthor={Youwen Xu}}
\makeatletter
\renewcommand{\@setauthors}{%
  \begingroup
  \trivlist\centering\item\relax
    {\normalfont\normalsize\bfseries\authors\par}%
    \vskip .35em
    {\normalfont\small East China University of Science and Technology\par}%
  \endtrivlist
  \endgroup
}
\let\shortauthors\@empty
\makeatother

\date{October 27, 2026}
\begin{document}

\begin{abstract}
We prove that every graph with \(m>0\) edges contains at most
\[
 \frac{3}{2\sqrt{2}}\left(\frac{m}{5}\right)^{5/2}
\]
induced pentagons. Equality holds precisely for balanced blow-ups of the
Clebsch graph, together with isolated vertices. The bound and equality
classification extend to arbitrary nonnegative vertex weights and to all
graphons, including fractional graphons. Moreover, a deficit of
\(\varepsilon\) in the normalized ratio can be repaired by
\(O(\sqrt{\varepsilon}\,m)\) edge edits on the original vertex set,
producing an exactly balanced Clebsch blow-up and isolates. The exponent
\(1/2\) is optimal. We further separate the linear cost of adjacency errors
from the quadratic cost of class imbalance. This yields an arithmetic
description of the finite-edge deficit and its explicit second-order term,
with a uniform power-saving remainder. For each fixed integer \(j\), the
exact extremal count at \(m=40k^2+j\) is eventually an integer coefficient
polynomial in \(k\), computable by a finite enumeration; we do not quantify
its threshold of validity. The proof combines global variational conditions,
an active-density bound, and an exact rational certificate on the
\(1{,}044\) graphs of order seven. A positive-slope identity permits a
single rational reference density. An exact conditional span excludes
triangles at equality, after which degree and codegree identities determine
the Clebsch structure. All finite certificates admit independent
verification using integer arithmetic.
\end{abstract}


\maketitle

\section{Introduction and results}

The edge version of inducibility asks how many induced copies of a fixed
graph can occur in a graph with a prescribed number of edges. Recent work
of Wang, Zhao, and Lu~\cite{Wang2026} and of Chao, Cohen Antonir, Li, and
Yu~\cite{Chao2025} studies this problem through entropy inequalities and
local directed graphs. For the pentagon, both works conjecture that balanced
blow-ups of \(C_5\) give the asymptotically optimal count
\((m/5)^{5/2}\); see \cite[Conjecture~8]{Wang2026} and
\cite[Conjecture~1.7]{Chao2025}. We determine the sharp constant, which is
larger by a factor of \(3/(2\sqrt{2})\), and prove the corresponding rigidity
and stability statements. We then resolve the second-order arithmetic
oscillation and obtain an eventual exact polynomial at each fixed offset
from the equality sequence.

The extremal construction itself is known. Davey, Hurley, de Joannis de
Verclos, Kang, and Volec~\cite[Lemmas~1.4, 7.1, and 7.2]{Davey2026} record
the Clebsch pentagon count and its balanced blow-ups while studying
triangle-free pentagon counts normalized by \(|G|\Delta(G)^4\). Our upper
bound concerns arbitrary graphs with a prescribed edge count. We include
the elementary construction count below to fix the normalization.

With the number of vertices prescribed, the induced-pentagon problem has
a different extremal structure. Balogh, Hu, Lidick\'y, and
Pfender~\cite{Balogh2016} established the asymptotic result using iterated
\(C_5\) blow-ups; Lidick\'y, Mattes, and Pfender~\cite{Lidicky2023}
subsequently determined the extremal count for every order. Those results
provide context for the induced-count problem. Here the active vertex count
is part of the optimization and may be accompanied by an arbitrary number
of isolated vertices.

Let \(H\) be a finite simple graph with nonnegative vertex weights
\(x=(x_v)_{v\in V(H)}\). Define
\[
 E(H,x)=\sum_{uv\in E(H)}x_ux_v,\qquad
 C(H,x)=\sum_{\substack{U\subseteq V(H)\\ H[U]\cong C_5}}
                  \prod_{v\in U}x_v .
\]
Thus all pentagons counted by \(C\) are induced, including when \(H\)
contains triangles. Set
\[
 c_*=\frac{3}{50\sqrt{10}}=\frac{192}{40^{5/2}},
 \qquad
 R_*=\frac{3}{2\sqrt{2}}=5^{5/2}c_*.
\]
The Clebsch graph is the Cayley graph on \(\mathbb F_2^4\) generated by
\[
 \{e_1,e_2,e_3,e_4,e_1+e_2+e_3+e_4\}.
\]
It has degree five and forty edges. The only nonempty linear dependence
among these five generators is their full sum. Every closed walk of length
five therefore uses all five generators once and is a simple pentagon.
There are \(16\cdot 5!/10=192\) pentagons, and there are no triangles.
Every pentagon is induced.

\begin{theorem}[Sharp weighted bound and equality]\label{thm:weighted}
Every finite simple graph with nonnegative vertex weights satisfies
\begin{equation}
 C(H,x)\le c_*E(H,x)^{5/2}.\tag{1}\label{eq:weighted}
\end{equation}
If \(E(H,x)>0\), equality holds if and only if, after zero-weight vertices
are discarded, the graph consists of a Clebsch blow-up and isolated vertices,
with the sixteen Clebsch classes having equal total weight.
\end{theorem}

Write \(I_5(G)\) for the number of induced pentagons in an ordinary graph.
A graph with \(m>0\) edges satisfies
\[
 I_5(G)\le R_*\left(\frac m5\right)^{5/2}.
\]
Equality holds exactly for a Clebsch blow-up with the same positive integer
number \(k\) of vertices in every class, together with isolates. Such a
graph has \(40k^2\) edges and \(192k^5\) induced pentagons.

\begin{corollary}[Asymptotic edge inducibility]\label{cor:asymptotic}
If \(\rho(C_5,m)\) is the largest number of induced pentagons in a graph
with exactly \(m\) edges, then
\[
 \rho(C_5,m)=c_*m^{5/2}+O(m^2).
\]
\end{corollary}
\begin{proof}
The upper bound is Theorem~1.1. For the lower bound, take a balanced
Clebsch blow-up with \(k=\lfloor\sqrt{m/40}\rfloor\) and add a disjoint
matching with \(m-40k^2\) edges. It has exactly \(m\) edges and \(192k^5\)
induced pentagons. Since \(m-40k^2=O(\sqrt m)\), the asserted error follows.
\end{proof}

In the normalization of~\cite{Chao2025}, our result reads
\[
 \operatorname{eind}(C_5)
 =\limsup_{m\to\infty}\frac{10\rho(C_5,m)}{(2m)^{5/2}}
 =\frac{3}{40\sqrt5}.
\]
The upper bound in \cite[Theorem~1.8]{Chao2025} is \(1/\sqrt{250}\) in this
normalization. Thus the sharp bound also settles the pentagon case
identified as open in that paper.

For a graphon \(W:[0,1]^2\to[0,1]\), use the normalization
\[
 \begin{gathered}
 E(W)=\frac12t(K_2,W),\qquad
 C(W)=\frac1{10}t_{\mathrm{ind}}(C_5,W),\\
 R(W)=5^{5/2}\frac{C(W)}{E(W)^{5/2}}.
 \end{gathered}
\]
Let \(\mathcal E_{\mathrm{Cl}}\) be the family obtained by partitioning an
active set of arbitrary positive measure into sixteen equal classes,
placing Clebsch adjacency between them, and making all remaining points
isolated.

\begin{theorem}[Graphon bound and equality]\label{thm:graphon}
Every positive-density graphon satisfies \(R(W)\le R_*\).
Equality holds exactly for \(\mathcal E_{\mathrm{Cl}}\), up to a measurable
partition and null sets. In particular every equality graphon is
\(\{0,1\}\)-valued.
\end{theorem}

The stability statement is on the edge scale. If
\(I_5(G)\ge (R_*-\varepsilon)(m/5)^{5/2}\), then
\(O(\sqrt{\varepsilon}\,m)\) edge additions and deletions suffice to obtain
an exactly balanced Clebsch blow-up and isolates on the original vertex
set. Theorem~8.1 states this precisely. There is no bound on the original
number of vertices and no additional rounding error. The exponent \(1/2\)
cannot be improved.

Our stronger structural conclusion distinguishes adjacency errors from
class imbalance. A graph with normalized deficit \(\varepsilon\) can be
edited into a possibly unbalanced Clebsch blow-up with integer class sizes
\(n_i\), total active size \(N\), and
\[
 |E(G)\mathbin{\triangle}E(H)|
       +\sum_{i=1}^{16}(n_i-N/16)^2=O(\varepsilon m).
\]
There is no additive rounding loss. We also prove a two-sided comparison
between the pentagon deficit and this structural quantity; see Section~9.

The integer edge count produces a nonconstant second-order term. Writing
\[
 \alpha_m=\left\{\sqrt{\frac{32m}{5}}\right\},\qquad
 f(\alpha)=\frac{(1-\alpha)\min\{3\alpha,1\}}{80},
\]
Theorem~12.2 gives the uniform expansion
\begin{equation}
 \rho(C_5,m)=c_*m^{5/2}-m^2f(\alpha_m)+O(m^{15/8}).\tag{2}
\end{equation}
In particular,
\[
 \limsup_{m\to\infty}\frac{c_*m^{5/2}-\rho(C_5,m)}{m^2}
 =\frac1{120}.
\]
The phase function comes from two different constructions: a partial clone
of a balanced template, and an integer rounding of a perturbation in the
negative eigenspace of the Clebsch matrix. The lower bound reduces every
extremal graph to the same scalar profile problem. Theorem~10.1 additionally
determines the order of the deficit uniformly at phases approaching zero,
where the absolute remainder in (2) need not be smaller than the displayed
second-order term. At every fixed integer offset \(m=40k^2+j\),
Theorem~13.1 gives the finer expansion
\(c_*m^{5/2}-\rho(C_5,m)=\Lambda(j)k^3+O_j(k^2)\), where \(\Lambda(j)\) is
an explicit finite lattice optimization. Theorem~14.1 then gives an
eventually exact integer coefficient polynomial for every fixed offset,
together with a finite procedure determining all eventual extremal
constructions. For example, Corollary~13.2 proves, for all sufficiently
large \(k\),
\[
 \rho(C_5,40k^2+1)=192k^5,\qquad
 \rho(C_5,40k^2-1)=192k^5-24k^3.
\]

Our certificate uses the flag algebra framework of Razborov~\cite{Razborov2007}.
Strong and perfect blow-up stability from flag algebra calculations has a
general fixed-order framework due to Pikhurko, Slia\v{c}an, and
Tyros~\cite{Pikhurko2019}. Bodn\'ar and Pikhurko~\cite{BodnarPikhurko2026}
combine exact flag algebra bounds with stability and integer optimization
of blow-up part sizes in several inducibility problems. Their
$(\kappa,\ell)$-edge-inducibility objective counts $\kappa$-vertex subsets
with exactly $\ell$ edges in a graph of fixed order; it is different from
prescribing the total edge count of the host graph, as done here.
The systematic study of six-vertex inducibility by Bodn\'ar, Gao,
Le\'on, Liu, Pikhurko, and Sun~\cite{BodnarEtAl2026Six} provides further
instances of sharp flag algebra bounds and perfect stability.
Thus the use of exact certificates and the reduction of an extremal
problem to integer part-size optimization are established techniques.
The estimates here are measured by $m$, allow the active mass to vary,
and control integer class variance separately from adjacency errors.
These features are retained in the arithmetic analysis, which also allows
bounded exceptional edge modifications at each fixed offset.

Automated numerical optimization and conversion to exact proofs are
available in Bodn\'ar's FlagAlgebraToolbox~\cite{Bodnar2026Toolbox}.
The computations for this paper use the separately supplied generators
and exact verifiers; FlagAlgebraToolbox is cited as a software and
methodological precedent, not as a dependency of these computations.

The proof first constructs a global graphon maximum. Global optimality
supplies vertex, edge, and new-profile variations; these are stronger than
weight optimality on one fixed template. A density bound allows every
relevant global maximum, including a fractional one, to be padded to
\(p_*=5/144\). An exact seven-vertex identity then proves the upper bound
at that density. A second exact span excludes triangles at equality.
Restrictions of the same positive-square vectors then give the degree
and codegree structure and identify the Clebsch graph. Finally, peeling,
compactness, local \(L^1\) stability, and whole-vertex rounding give the
quantitative edit bound. Keeping the linear adjacency defect and the
quadratic mass defect separate then permits the finite-edge arithmetic
analysis.

\begin{figure}[tbp]
 \centering
 \includegraphics[width=.92\textwidth]{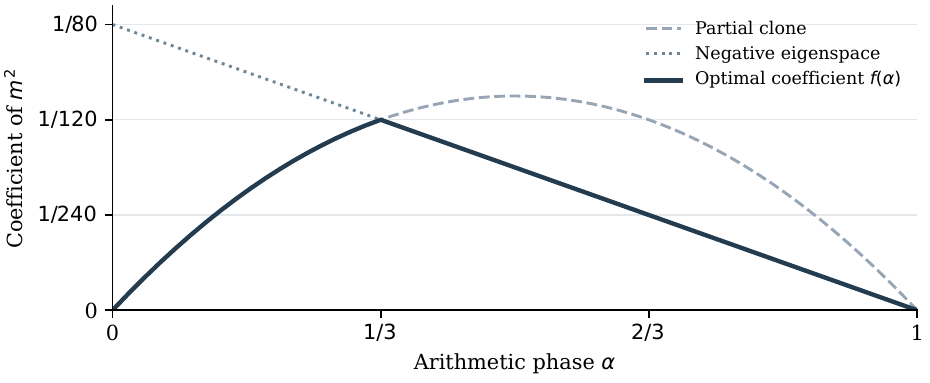}
 \caption{The second-order coefficient and the two constructions that
 attain it. Their transition is at \(\alpha=1/3\), where
 \(f(\alpha)=1/120\).}
 \label{fig:arithmetic-phase}
\end{figure}

\section{Global extremizers and their variations}

It is useful to allow a finite vertex measure \(\mu\), of total mass
\(s=\mu(\Omega)\), and a symmetric measurable kernel
\(W:\Omega^2\to[0,1]\). All densities in this section are integrated against
\(\mu\). Thus
\begin{align*}
 E&=\frac12\int W(x,y)\,d\mu(x)\,d\mu(y),\\
 C&=\frac1{10}\int
      \prod_{ij\in E(C_5)}W(x_i,x_j)
      \prod_{\substack{i<j\\ij\notin E(C_5)}}(1-W(x_i,x_j))\,d\mu^5.
\end{align*}
Multiplying \(\mu\) by a constant does not change \(C/E^{5/2}\).
Write \(d(x)=\allowbreak\int W(x,y)\,d\mu(y)\) and let \(L(x)\) be the first
derivative of \(C\) with respect to vertex mass at \(x\). Equivalently,
\(L(x)\) is one half of the induced rooted pentagon integral on four
further vertices. A rooted pentagon is always induced throughout the paper.

\subsection{Finite templates and a compact maximization problem}

\begin{lemma}[Finite-template attainment]\label{lem:finite-attainment}
For a fixed finite graph containing an induced pentagon, the supremum of
\(C(H,x)/E(H,x)^{5/2}\) over nonnegative weights with \(E(H,x)>0\) is
attained on a positive-weight induced subgraph.
\end{lemma}
\begin{proof}
Normalize \(E=1\). The square of a pentagon monomial is the product of its
five cyclic edge weights, and is at most \(5^{-5}\) by AM--GM. There are
finitely many such monomials, so the supremum is finite. Take a maximizing
sequence and pass to a subsequence on which every vertex weight converges
in \([0,\infty]\).

If a pentagon contains a weight \(M\to\infty\), its two cyclic neighbors
have weights at most \(1/M\). Its remaining two vertices are adjacent, so
their product is at most one. The pentagon monomial is consequently at
most \(1/M\) and tends to zero. A monomial with only finite limiting
coordinates, one of them zero, also tends to zero. The finite positive
limiting coordinates therefore retain the entire limiting pentagon mass.
Their edge mass is at most one and is positive. If it were strictly less
than one, their ratio would exceed the original supremum. Hence it equals
one, proving attainment.
\end{proof}

\begin{lemma}[A global graphon maximum]\label{lem:global-maximum}
The supremum of \(R\) over all finite graphs, over all finite nonnegative
weighted graphs, and over all positive-density graphons is the same finite
number. It is attained by a graphon of edge density \(1/80\).
\end{lemma}
\begin{proof}
The finite ordinary and weighted suprema agree. Indeed, rational weights
are realized by integer independent blow-ups, and an induced \(C_5\)
cannot use two vertices of one blow-up class: those vertices would be
false twins, whereas \(C_5\) has no false twins. Approximation by rational
weights gives the claim.

At a positive finite-template maximum, remove isolates and normalize the
weights to sum one. Put
\[
 p=2E,\qquad \theta=10C,\qquad
 \kappa=\frac{5C}{2E}=\frac{\theta}{2p}.
\]
The first variation is \(C_i=\kappa d_i\). The elementary universal link
inequality
\begin{equation}
 C_i\le\frac14Ed_i^2\tag{3}
\end{equation}
is proved, including its fractional version, in Lemma~3.1. Since all
active degrees are positive, it implies \(d_i\ge4\theta/p^2\). Averaging
yields
\begin{equation}
 p\ge\frac{16R^2}{1000},\qquad
 R=\sqrt{1000}\,\frac{\theta}{p^{5/2}}.\tag{4}
\end{equation}
In particular the common finite supremum \(M\) is at most
\(\sqrt{1000}/4\), and it is at least \(R_*\) by the Clebsch construction.

Every positive-density graphon has ratio at most \(M\). To see this,
sample independent latent points and then sample edges independently with
the specified probabilities. For any fixed induced pattern, the normalized
injective count converges in probability to its graphon density. The
variance is \(O(1/N)\): pairs of injections on disjoint vertex sets are
independent, and overlapping pairs form an \(O(1/N)\) fraction. Apply this
simultaneously to the edge and induced pentagon counts. The edge density
has a positive limit, so the ratios converge as well.

Conversely, every finite graph with ratio at least \(R_*\) can first be
optimized on its finite template. By (4), the resulting active probability
weighting has edge density at least \(9/500>1/80\). Rescale its active
measure and add isolates to obtain edge density \(1/80\), without changing
its ratio. Thus the supremum over graphons of this fixed density is also
\(M\).

We use the compactness of graphons modulo cut distance
\cite[Theorem~5.1]{LovaszSzegedy2007} and the continuity of fixed
homomorphism densities~\cite{LovaszSzegedy2006}. The induced pentagon
density is continuous as well: expand its five nonedge factors into a
finite alternating sum of ordinary homomorphism densities. The class of
graphons of edge density \(1/80\) is closed and compact, and the
denominator of \(R\) is fixed there. It therefore contains a maximizer.
The sampling argument shows that this is a global maximum over all
positive edge densities, not only over the fixed-density class.
\end{proof}

This globality is essential. Edge variations and insertion of arbitrary
new vertex profiles cannot be inferred merely from optimality of the
weights on a fixed finite template.

\subsection{Vertex and edge variations}

Let \(A(x,y)\) be the induced pentagon completion kernel after one cyclic
edge factor is deleted, and let \(B(x,y)\) be the completion kernel after
one chord nonedge factor is deleted. For explicit conventions, suppressing
the measure symbols,
\begin{align*}
 A(x,y)&=\int W(x,a)W(a,b)W(b,c)W(c,y)\,N_A\,da\,db\,dc,\\
 N_A&=(1-W(x,b))(1-W(x,c))(1-W(y,a))\\
    &\hspace{2em}\cdot(1-W(y,b))(1-W(a,c)),\\
 B(x,y)&=\int W(x,a)W(a,y)W(y,b)W(b,c)W(c,x)\,N_B\,da\,db\,dc,\\
 N_B&=(1-W(x,b))(1-W(a,b))(1-W(a,c))(1-W(y,c)).
\end{align*}
Both kernels are symmetric. Put \(q=A-B\).

\begin{lemma}[Necessary conditions at a global maximum]\label{lem:variations}
At a positive global graphon maximum of \(C/E^{5/2}\), define
\[
 \kappa=\frac{5C}{2E},\quad C''=WA+(1-W)B,\quad
 M=C''-\kappa W,\quad G=\frac{3\kappa}{2E}d\otimes d-M.
\]
Then
\begin{gather}
 L(x)=\kappa d(x)\qquad\text{almost everywhere},\tag{5}\\
 G\succeq0\qquad\text{as an integral operator on \(L^2(\mu)\)},\tag{6}\\
 W(q-\kappa)\ge0,\quad (1-W)(q-\kappa)\le0
       \quad\text{almost everywhere}.\tag{7}
\end{gather}
Consequently
\begin{equation}
 \begin{gathered}
 M\ge B\ge0,\qquad
 \int W(x,y)M(x,y)\,d\mu(y)=\kappa d(x)\\
 \text{almost everywhere}.
 \end{gathered}\tag{8}
\end{equation}
\end{lemma}
\begin{proof}
Replace \(d\mu\) by \((1+sh)\,d\mu\), with \(h\) bounded and \(s\) of
either sign sufficiently small. Normalizing the total measure afterward
does not change the ratio. The first derivative gives (5); the second gives
(6), first for bounded \(h\) and then for all \(h\in L^2(\mu)\) by density.
The kernel of the second vertex derivative is \(WA+(1-W)B\): the two
selected cycle vertices are adjacent in ten ordered positions and
nonadjacent in ten ordered positions, canceling the factor \(1/10\) in
\(C\).

For a bounded symmetric edge perturbation \(H\),
\[
 DC(W)[H]=\frac12\int q(x,y)H(x,y)\,d\mu^2,\qquad
 DE(W)[H]=\frac12\int H(x,y)\,d\mu^2.
\]
Use the feasible one-sided directions \(H=-Wf\) and \(H=(1-W)f\), with
\(f\ge0\). Global maximality yields (7). In particular
\(W(1-W)(q-\kappa)=0\) almost everywhere.

The identity \(M=B+W(q-\kappa)\) proves \(M\ge B\). Furthermore
\[
 \int W(x,y)A(x,y)\,d\mu(y)=2L(x).
\]
Using \(C''=WA+(1-W)B\) and the preceding complementary slackness
identity, we obtain
\begin{align*}
 \int WM
 &=\int WA-\int W(1-W)q-\kappa\int W^2\\
 &=2L-\kappa d=\kappa d.
\end{align*}
This proves (8), including on fractional parts of the graphon.
\end{proof}

On a probability space write
\[
 p=t(K_2,W),\qquad\theta=t_{\mathrm{ind}}(C_5,W).
\]
Multiplying (5) by \(d\), \(d^2\), and
\(b(x)=\int W(x,y)d(y)\,dy\), respectively, gives
\begin{equation}
 p\,t(L_j,W)=\theta\,t(B_j,W),\qquad j=1,2,3.\tag{9}
\end{equation}
Here \(L_1,L_2,L_3\) are an induced pentagon with, respectively, a leaf,
two leaves at the same root, and a pendant path of length two at that root.
Only the five original cycle vertices are required to induce a pentagon;
all other unmentioned pairs are unrestricted. The graphs \(B_1,B_2,B_3\)
are \(P_3,K_{1,3}\), and \(P_4\), with ordinary homomorphism density
conventions.

\subsection{Insertion of a new vertex profile}

For a measurable function \(f:\Omega\to[0,1]\), write \(d(f)=\int f\) and
\begin{equation}
 L(f)=\frac12\int I_{P_4}(a,b,c,d)
                 f(a)f(d)(1-f(b))(1-f(c))\,d\mu^4,\tag{10}
\end{equation}
where
\[
 I_{P_4}=W(a,b)W(b,c)W(c,d)
           (1-W(a,c))(1-W(a,d))(1-W(b,d)).
\]

\begin{lemma}[Profile insertion]\label{lem:profile-insertion}
At a positive global maximum, \(L(f)\le\kappa d(f)\) for every
measurable \(f:\Omega\to[0,1]\).
\end{lemma}
\begin{proof}
Add a new independent class of mass \(\epsilon\), with adjacency profile
\(f\) to the old space. Before normalizing its total measure, the new masses
are
\[
 E_\epsilon=E+\epsilon d(f),\qquad
 C_\epsilon=C+\epsilon L(f)+O(\epsilon^2).
\]
The linear term selects exactly one new cycle vertex. If \(f\) is
fractional, cycles using two new vertices can occur, but their contribution
is still \(O(\epsilon^2)\). Differentiating the ratio in this feasible
direction proves the claim.
\end{proof}

In particular, put \(g=W(u,\cdot)\), \(h=W(v,\cdot)\), and \(f=g(1-h)\).
Combining the lemma with (5) gives
\begin{equation}
 J_\theta(u,v):=L(g)-L(f)
       -\frac{\theta}{2p}\bigl(d(g)-d(f)\bigr)\ge0\tag{11}
\end{equation}
on a probability space. Its removed degree \(d(g)-d(f)=\int gh\) is
nonnegative. Multiplying (11) by the nonnegative factor
\(W(u,v)W(u,z)W(v,z)\) and integrating gives the one profile condition
used in the certificate.

\begin{lemma}[Padding]\label{lem:padding}
Restrict a probability graphon to its positive-degree set and normalize
that restriction to have total measure one. If its edge density is
\(p_a>p_*\), give this normalized active graphon total measure
\(q=\sqrt{p_*/p_a}\) and adjoin isolate mass \(1-q\). The resulting graphon
has the original ratio \(R\) and edge density \(p_*\). If the original
graphon is a global ratio maximum, so is the padded one, and all the
preceding necessary conditions apply.
\end{lemma}
\begin{proof}
Edge mass and pentagon mass scale by the second and fifth powers of the
active measure, respectively. Their normalized ratio is unchanged.
Globality follows from equality of the ratio with the global supremum,
rather than from any restricted variational argument on the original
template.
\end{proof}

\section{A lower bound on the active density}

We work temporarily on a finite measure space \((\Omega,\mu)\) and write
\(S=\mu(\Omega)\). The kernel \(W\) always takes values in \([0,1]\);
rescaling the vertex measure changes neither the ratio \(C/E^{5/2}\) nor
its global maximality. Throughout this section,
\[
 \begin{gathered}
 d(x)=\int_\Omega W(x,y)\,d\mu(y),\qquad
 b(x)=\int_\Omega W(x,y)d(y)\,d\mu(y),\\
 S_2=\int_\Omega d(x)^2\,d\mu(x).
 \end{gathered}
\]
All integrals below are with respect to the indicated powers of \(\mu\).
The vertex derivative of \(C\) is denoted by \(L_x\), and
\(\kappa=5C/(2E)\).

\begin{lemma}[Universal link bounds]\label{lem:universal-links}
For every \([0,1]\)-valued graphon on a finite measure space, set
\begin{align*}
 g_x(u)&=W(x,u),\\
 s_x&=\int_{\Omega^2}g_x(u)(1-g_x(v))W(u,v)\,d\mu(u)\,d\mu(v),\\
 E_{\mathrm{out}}(x)&=\frac12\int_{\Omega^2}
       (1-g_x(u))(1-g_x(v))W(u,v)\,d\mu(u)\,d\mu(v).
\end{align*}
Then
\begin{equation}
 L_x\le\frac{d(x)^2}{4}E_{\mathrm{out}}(x)
       \le\frac{E\,d(x)^2}{4},\tag{12}
\end{equation}
and
\begin{equation}
 L_x\le\frac{s_x^2}{2},\qquad
 E_{\mathrm{out}}(x)=E-\frac{b(x)+s_x}{2}.\tag{13}
\end{equation}
\end{lemma}
\begin{proof}
Fix \(x\) and abbreviate \(g_x\) to \(g\). If
\[
 \Phi(a,b,c,e)=W(a,b)W(b,c)W(c,e)
       (1-W(a,c))(1-W(a,e))(1-W(b,e)),
\]
then the exact rooted link formula is
\[
 L_x=\frac12\int_{\Omega^4}
       \Phi(a,b,c,e)g(a)g(e)(1-g(b))(1-g(c)).
\]
The factor \(1/2\) identifies the two orientations of the rooted pentagon.
For a fixed opposite edge \(bc\), the two endpoint choices have masses
\begin{align*}
 U_{bc}&=\int_\Omega g(a)W(a,b)(1-W(a,c))\,d\mu(a),\\
 V_{bc}&=\int_\Omega g(a)W(a,c)(1-W(a,b))\,d\mu(a).
\end{align*}
The elementary inequality \(u(1-v)+v(1-u)\le1\) for \(u,v\in[0,1]\)
gives \(U_{bc}+V_{bc}\le d(x)\). Dropping the nonedge factor between
the two endpoints and using \(U_{bc}V_{bc}\le d(x)^2/4\) proves the first
inequality in (12). The second follows from \(E_{\mathrm{out}}(x)\le E\).

For the other bound, retain in the link integrand only
\[
 g(a)(1-g(b))W(a,b)\;g(e)(1-g(c))W(e,c).
\]
All discarded factors lie in \([0,1]\). The two remaining pairs integrate
independently, giving \(L_x\le s_x^2/2\). Finally, with
\[
 t_x=\frac12\int_{\Omega^2}g(u)g(v)W(u,v),
\]
we have \(b(x)=2t_x+s_x\) and
\(E_{\mathrm{out}}(x)=E-b(x)+t_x\). These are exactly the last identity
in (13).
\end{proof}

The next estimate uses global maximality. We recall the notation from the
variation formulas: \(A\) and \(B\) are the completions of an induced
pentagon after deleting, respectively, the root edge and the root nonedge.
Put
\[
 M=C''-\kappa W=B+W(A-B-\kappa).
\]
At a global maximizer the first-order, edge, and second-order conditions
give
\begin{equation}
 \begin{gathered}
 L_x=\kappa d(x),\qquad M(x,y)\ge B(x,y)\ge0,\\
 \int_\Omega W(x,y)M(x,y)\,d\mu(y)=\kappa d(x),\\
 G=\frac{3\kappa}{2E}d\otimes d-M\succeq0.
 \end{gathered}\tag{14}
\end{equation}
In particular, after zero-degree points are discarded, Lemma~3.1 implies
\begin{equation}
 d(x)\ge\frac{4\kappa}{E}=\frac{10C}{E^2}
       \quad\text{almost everywhere}.\tag{15}
\end{equation}

\begin{lemma}[A second-order moment inequality]\label{lem:second-order-moment}
At a global maximizer with \(C>0\), on its active finite measure space,
\begin{equation}
 S_2\ge\frac{ES}{3}+\frac{20C}{E}.\tag{16}
\end{equation}
\end{lemma}
\begin{proof}
Write
\[
 \gamma(x,y)=\int_\Omega W(x,z)W(y,z)\,d\mu(z)
\]
for codegree, and let
\[
 Q(x,y)=\int_{\Omega^2}
 W(x,a)W(a,b)W(b,y)(1-W(x,b))(1-W(a,y))\,d\mu(a)\,d\mu(b).
\]
Thus \(Q\) is the completion of an induced three-edge path with its
endpoint nonedge factor omitted. In \(B(x,y)\), delete the two nonedge
factors joining the common neighbor to the internal vertices of this path.
This gives
\[
 B(x,y)\le\gamma(x,y)Q(x,y).
\]
Define
\[
 D=\int_{\Omega^2}(1-W)Q=t_{\mathrm{ind}}(P_4,W),\qquad
 L_0=\int_{\Omega^2}(1-W)B\gamma.
\]
Since \(\int(1-W)B=10C\), Cauchy--Schwarz yields
\begin{equation}
 100C^2\le DL_0.\tag{17}
\end{equation}
Here \(B/\gamma\) may be assigned the value zero where \(\gamma=0\);
the bound \(B\le\gamma Q\) makes this convention harmless.

We also have
\begin{equation}
 D\le E^2.\tag{18}
\end{equation}
Indeed, fix the two outer edges on four independent points. Of the four
possible cross edges, the probabilities that exactly one is present sum
to at most one. Each choice, after integration, is the same induced-\(P_4\)
density \(D\). Consequently \(4D\le(\int W)^2=4E^2\). This proof applies
directly to fractional values of \(W\).

Let \(T\) be the integral operator with kernel \(W\). We now justify an
optimized diagonal test of the positive operator \(G\) in (14). Fix a
bounded real function \(z\), and choose increasing finite measurable
partitions whose conditional expectations of \(W,M,d,z\) converge in
\(L^2\), using the product partitions for the two kernels. Such partitions
exist by approximating the kernels by finite linear combinations of
measurable rectangles and the functions by simple functions, then
successively refining all the sets that occur. Let \(P_n\) denote
orthogonal projection onto the associated step functions, and set
\[
 T_n=P_nTP_n,\qquad M_n=P_nMP_n,\qquad d_n=P_nd.
\]
These are compressions of the original kernels, not derivatives recomputed
at an approximating graphon. In particular
\[
 G_n=\frac{3\kappa}{2E}d_n\otimes d_n-M_n\succeq0.
\]
Put \(z_n=P_nz\) and let \(Z_n\) be
multiplication by \(z_n\) on the finite-dimensional step space. The matrix
\((T_n-Z_n)^2\) is positive semidefinite, so
\[
 \operatorname{tr}\bigl(M_n(T_n-Z_n)^2\bigr)
       \le\frac{3\kappa}{2E}\|(T_n-Z_n)d_n\|_2^2.
\]
In the orthonormal basis \(\mathbf1_U/\sqrt{\mu(U)}\) of partition cells,
each diagonal entry of \(M_n\) is nonnegative because \(M\ge0\).
Therefore \(\operatorname{tr}(M_nZ_n^2)\ge0\), and discarding this term
gives
\begin{equation}
 \operatorname{tr}(M_nT_n^2)
 \le\frac{3\kappa}{2E}\|T_nd_n-Z_nd_n\|_2^2
                 +2\operatorname{tr}(M_nT_nZ_n).\tag{19}
\end{equation}
We explain all three limits. Conditional expectations of the bounded
kernels converge in Hilbert--Schmidt norm, so
\(\operatorname{tr}(M_nT_n^2)\to\operatorname{tr}(MT^2)\). Also
\(T_nd_n-Z_nd_n\to b-zd\) in \(L^2\): the kernels, \(d_n\), and \(z_n\)
are uniformly bounded, and the conditional expectations converge in
\(L^2\). Finally, if the same symbols denote the step kernels, the mixed
trace is
\[
 \operatorname{tr}(M_nT_nZ_n)
   =\int_{\Omega^2}z_n(x)M_n(x,y)W_n(x,y)\,d\mu(x)\,d\mu(y).
\]
The integrands converge in \(L^1\), and (14) identifies the limit as
\(\kappa\int zd\). No infinite-dimensional trace involving the arbitrary
multiplication operator \(Z^2\) is required; its nonnegative contribution
was discarded before taking a limit.

The kernel of \(T^2\) is \(\gamma\ge0\). Since \(M\ge B\),
\[
 \operatorname{tr}(MT^2)=\int_{\Omega^2}M\gamma
       \ge\int_{\Omega^2}(1-W)B\gamma=L_0.
\]
Passing to the limit in (19) gives
\begin{equation}
 L_0\le\kappa\int_\Omega
       \left[\frac{3}{2E}(b-zd)^2+2zd\right]\,d\mu.\tag{20}
\end{equation}
Choose
\[
 z(x)=\frac{b(x)-2E/3}{d(x)}.
\]
This is bounded: \(b/d\le S\), and the positive lower bound (15)
controls \(1/d\) on the active space. The choice minimizes the scalar
quadratic in (20), yielding
\[
 L_0\le\kappa\left(2\int_\Omega b\,d\mu-\frac{2ES}{3}\right)
       =2\kappa\left(S_2-\frac{ES}{3}\right).
\]
Combine this with (17) and (18), and use \(\kappa=5C/(2E)\).
The result is (16).
\end{proof}

\begin{proposition}[Active density]\label{prop:active-density}
Let \(W\) be a global maximizer of \(C/E^{5/2}\) among graphons of
positive edge density. Suppose
\[
 5^{5/2}\frac{C}{E^{5/2}}\ge\frac{3}{2\sqrt2}.
\]
After its zero-degree set is deleted and the remaining measure is
normalized to one, its edge density satisfies
\begin{equation}
 p_{\mathrm{act}}\ge
 \frac{80000}{(491\sqrt{10}-75\sqrt2)^2}>\frac5{144}.\tag{21}
\end{equation}
\end{proposition}
\begin{proof}
Delete zero-degree points, which changes neither \(E\) nor \(C\), and
rescale the measure so that \(E=1\). Write \(c=C\) in this normalization,
so \(c\ge c_*=3/(50\sqrt{10})\) and \(\kappa=5c/2\). By (15),
\begin{equation}
 d(x)\ge\delta:=10c,\qquad S\le\frac1{5c}.\tag{22}
\end{equation}
The second inequality uses \(\int d=2\).

Put
\[
 T_0=\int_\Omega d(x)b(x)\,d\mu(x),\qquad
 V=\int_\Omega d(x)s_x\,d\mu(x).
\]
The inequality
\(d(x)d(y)\ge\delta(d(x)+d(y))-\delta^2\), integrated against \(W(x,y)\),
and Lemma~3.2 give
\begin{equation}
 T_0\ge2\delta S_2-2\delta^2
       \ge\frac{20}{3}cS+200c^2.\tag{23}
\end{equation}
Here \(\delta\) is the fixed bound \(10c\), so no substitution is made in
a nonmonotone function of the actual minimum degree. By
\(L_x=\kappa d(x)\) and (13), \(s_x\ge\sqrt{2\kappa d(x)}\).
Jensen's inequality therefore gives
\begin{equation}
 V\ge\sqrt{2\kappa}\int_\Omega d(x)^{3/2}\,d\mu(x)
       \ge4\sqrt{\frac{\kappa}{S}}.\tag{24}
\end{equation}
Divide the first inequality in (12) by \(d(x)>0\), use
\(L_x=\kappa d(x)\), and integrate. With
\(E_{\mathrm{out}}(x)=1-(b(x)+s_x)/2\) this reads
\[
 8\kappa S+T_0+V\le4.
\]
Equations (23) and (24) now imply
\[
 \frac{40}{3}cS+100c^2+2\sqrt{\frac{\kappa}{S}}\le2.
\]
The mass bound in (22) gives \(2\sqrt{\kappa/S}\ge5\sqrt2\,c\), and hence
\[
 S\le\frac3{20c}-\frac{15}{2}c-\frac{3\sqrt2}{8}.
\]
The right-hand side is strictly decreasing for \(c>0\). Substituting
\(c_*=3/(50\sqrt{10})\) yields
\[
 S\le\frac{491\sqrt{10}-75\sqrt2}{200}.
\]
Since \(11\sqrt{10}<75\sqrt2\), this upper bound is strictly less than
\(480\sqrt{10}/200=\sqrt{288/5}\). On normalizing the active measure to
one, its edge density is \(p_{\mathrm{act}}=2/S^2\). This proves both
inequalities in (21).
\end{proof}

In particular, a global maximizer can be padded by isolated mass to the
rational density \(p_*=5/144\). The same conclusion applies to any
equality graphon once the sharp bound has been established: such a graphon
is itself a global maximizer. Thus the padding needed for the equality
analysis does not require a finite-support or a \(\{0,1\}\)-valued
assumption.

\section{A single exact certificate}

We work on a probability space and write
\[
 p=t(K_2,W),\qquad\theta=t_{\mathrm{ind}}(C_5,W),\qquad
 p_*=\frac5{144},\qquad\tau=\frac5{663552}.
\]
Balanced Clebsch on an active set of measure \(1/3\), with isolates outside
it, has \((p,\theta)=(p_*,\tau)\). The certificate in this section applies
to arbitrary graphons of density \(p_*\); its variational interpretation
is used at a global maximum.

\subsection{The additional variational columns}

Retain the partially induced patterns \(L_j\) and the graphs \(B_j\) from
(9), and abbreviate
\[
 \Phi_j(s)=p_*t(L_j,W)-s\,t(B_j,W),\qquad j=1,2,3.
\]
We retain four additional columns. We describe them explicitly to specify
their signs and their dependence on the target density. Two have zero
multipliers in the supplied certificate.

First let \(z(x)=2L(x)\) be the rooted labeled induced-pentagon density,
and put
\[
 T(x)=\int W(x,y)W(x,w)W(y,w)\,dy\,dw,
\]
\[
 L_\triangle=\int z(x)T(x)\,dx,\qquad
 B_\triangle=\int d(x)T(x)\,dx.
\]
Equivalently, \(L_\triangle\) is the partially induced pattern obtained
by attaching a triangle at one vertex of an induced pentagon, and
\(B_\triangle\) is the ordinary density of a triangle with a leaf at its
root. Define
\[
 \Phi_\triangle(s)=p_*L_\triangle-sB_\triangle.
\]
Multiplying the first-order identity by \(T(x)\) gives
\(\Phi_\triangle(\theta)=0\), and \(B_\triangle\ge0\).

For the next two columns, fix ordered roots \(u,v\), write \(g=W(u,\cdot)\)
and \(h=W(v,\cdot)\), and set
\begin{align*}
 \psi_0(u,v)&=\int g(x)h(x)g(y)h(y)(1-W(x,y))\,dx\,dy,\\
 \psi_1(u,v)&=\int g(x)h(x)g(y)h(y)W(x,y)\,dx\,dy.
\end{align*}
Conditional on the root edge being present, these are the probabilities of
a diamond with its two degree-three vertices rooted, and of a complete
four-vertex graph. Using the completion kernels \(A,B\) from Section~2,
define
\[
 J_i(s)=\int W(u,v)\psi_i(u,v)
       \left(A(u,v)-B(u,v)-\frac{s}{2p_*}\right)\,du\,dv,
       \qquad i=0,1,
\]
and \(R_i=\int W(u,v)\psi_i(u,v)\,du\,dv\ge0\).
The existing-edge condition (7) gives \(J_i(\theta)\ge0\) at a global
maximum.

Finally, let \(J_{\mathrm{del}}(s)\) be the integral of the expression
in (11), with \(\theta\) replaced by \(s\), against
\(W(u,v)W(u,w)W(v,w)\,du\,dv\,dw\).
Thus \(J_{\mathrm{del}}(\theta)\ge0\). Its removed-degree coefficient is
\[
 D_{\mathrm{del}}=\int W(u,v)c(u,v)^2\,du\,dv\ge0,\qquad
 c(u,v)=\int W(u,x)W(v,x)\,dx.
\]
In particular \(D_{\mathrm{del}}=R_0+R_1\). At fixed edge density,
\begin{equation}
 \begin{aligned}
 \Phi_j(\tau)&=(\theta-\tau)t(B_j,W),
 &\Phi_\triangle(\tau)&=(\theta-\tau)B_\triangle,\\
 J_i(\tau)&=J_i(\theta)+\frac{\theta-\tau}{2p_*}R_i,
 &J_{\mathrm{del}}(\tau)&=J_{\mathrm{del}}(\theta)
                  +\frac{\theta-\tau}{2p_*}D_{\mathrm{del}}.
 \end{aligned}\tag{25}
\end{equation}
The first line uses the vertex first-order conditions. All added
coefficients of \(\theta-\tau\) in the second line are nonnegative.

\subsection{Seven-vertex coefficients and flag squares}

There are \(1044\) unlabeled simple graphs on seven vertices. Their
canonical adjacency masks, edge lists and orbit sizes are supplied with
the certificate. Write \(\pi_H(W)\) for the probability that seven
independent graphon samples induce \(H\). For a pattern on \(r\le7\)
labeled vertices, specify its required edges and nonedges; all other
pairs are unrestricted. Its coefficient on \(H\) is the number of
injections satisfying these requirements divided by \((7)_r\).
Averaging this coefficient against \(\pi_H(W)\) gives its graphon density
exactly. In particular the induced-pentagon coefficient is \(I_5(H)/252\).

The five rooted flag spaces are as follows.
\begin{center}
\begin{tabular}{lrrrr}
\toprule
Root type & Roots & Flag order & Flags & Denominator\\
\midrule
One root & 1 & 4 & 20 & 140\\
Three independent roots & 3 & 5 & 72 & 1260\\
Edge and isolated root & 3 & 5 & 72 & 1260\\
Path \(01,12\) & 3 & 5 & 72 & 1260\\
Triangle & 3 & 5 & 72 & 1260\\
\bottomrule
\end{tabular}
\end{center}
Root labels are fixed individually. Isomorphisms may permute only the
extension vertices. For each host, sample an ordered root embedding and
two disjoint extension sets. Cross-edges between the two extensions are
unrestricted. The sampling denominators are
\(7\binom63=140\) and \((7)_3\binom42=1260\).

For a flag vector \(v\), let \(q_F(x)\) be its conditional flag
probabilities at a root tuple and let \(P_\sigma(x)\) be the probability
of the root type. The associated square is
\begin{equation}
 Q_v(W)=\int P_\sigma(x)\left(\sum_Fv_Fq_F(x)\right)^2\,dx\ge0.\tag{26}
\end{equation}
The root-type factor occurs once. This convention is essential for
fractional graphons. An individual finite host matrix need not be positive
semidefinite; its average over \(\pi_H(W)\) is the matrix represented by
(26).

For clarity, the raw source counts use the injection denominator \((7)_r\)
for a pattern on \(r\) labeled vertices. Thus the cycle and house
completions in \(J_i\), and \(L_\triangle\), use \(5040\);
\(R_i\), \(B_\triangle\), and \(D_{\mathrm{del}}\) use \(840\).
The link difference in \(J_{\mathrm{del}}\) uses
\(10080=2\cdot7!\), including the \(1/2\) in the rooted link (10).
Internal symmetries permit cancellations in some of these fractions, but
the data files retain the raw injection counts. Each expanded profile term
is multilinear in the formal graph-edge variables, so Boolean-host
averaging is valid also for fractional \(W\).

\begin{proposition}[Exact finite identity]\label{prop:finite-certificate}
For every graphon of density \(p_*\) there is the identity
\begin{equation}
 \begin{aligned}
 &\tau-\theta-\sum_{j=1}^3\lambda_j\Phi_j(\tau)
       -\lambda_\triangle\Phi_\triangle(\tau)
       -\sum_{i=0}^1a_iJ_i(\tau)-a_{\mathrm{del}}J_{\mathrm{del}}(\tau)\\
 &\hspace{3em}
       =\sum_{r=1}^{184}b_rQ_{v_r}(W)+\sum_Hs_H\pi_H(W).
 \end{aligned}\tag{27}
\end{equation}
Here \(\lambda_\triangle,a_{\mathrm{del}},b_r\) are strictly positive
rational numbers, \(a_0=a_1=0\), and \(s_H\ge0\). Exactly \(826\) host
slacks are positive. The three unrestricted first-order multipliers
satisfy the exact inequalities
\begin{equation}
 -\frac{1157}{10}<\lambda_1<0,\qquad
 \lambda_3<0,\qquad\lambda_2+\lambda_3>310.\tag{28}
\end{equation}
All rational coefficients and integer flag vectors are specified in
\texttt{\upshape all\_\allowbreak graph\_\allowbreak exact\_\allowbreak certificate.json}.
\end{proposition}
\begin{proof}
The supplied identity has one row for each of the \(1044\) hosts.
Besides the displayed columns, its finite form includes the \(34\)
unrestricted fixed-density identities
\(p_F(W)t(K_2,W)-p_*p_F(W)\), one for each unlabeled graph \(F\) on
five vertices, where \(p_F\) is its induced-graph probability.
Their sum is \(t(K_2,W)-p_*\), so this basis is equivalent to that
single identity together with any \(33\) of the five-vertex identities.
In the finite coefficients the additional edge is sampled disjointly
from the vertices of \(F\). These columns vanish after graphon averaging
at density \(p_*\).

The checker computes the entire host atlas, all five flag matrices,
the partially induced patterns and all the additional columns from their
edge requirements. It then computes every declared source column and
verifies the \(1044\) rational identities and every required coefficient
sign exactly. Clearing one
common denominator reduces the final identity check to integer arithmetic.
The rational comparisons in (28) are checked in the same way. Thus
averaging the finite identities proves (27); the numerical search used
to find the coefficients is not part of the verification.
\end{proof}

\subsection{A uniformly positive slope}

Put \(x=t(P_3,W)\), \(y=t(K_{1,3},W)\) and \(z=t(P_4,W)\).
For every graphon,
\[
 z=\int W(u,v)d(u)d(v)\,du\,dv
 \le\frac12\int W(u,v)\bigl(d(u)^2+d(v)^2\bigr)\,du\,dv=y,
\]
and Cauchy--Schwarz gives \(x^2\le py\). Consequently
\begin{equation}
 \begin{aligned}
 1+\lambda_1x+\lambda_2y+\lambda_3z
 &\ge1+\lambda_1x+(\lambda_2+\lambda_3)y\\
 &\ge1+\lambda_1x+\frac{\lambda_2+\lambda_3}{p_*}x^2\\
 &\ge1-\frac{p_*\lambda_1^2}{4(\lambda_2+\lambda_3)}
 >\frac{2232551}{3571200}>\frac58.
 \end{aligned}\tag{29}
\end{equation}
The penultimate comparison uses (28).

\begin{corollary}[Bound at a global maximum]\label{cor:certificate-bound}
A positive global maximum of the induced-pentagon to edge ratio,
normalized to edge density \(p_*\), satisfies \(\theta\le\tau\).
If \(\theta=\tau\), all \(184\) ordinary squares in (27) vanish.
\end{corollary}
\begin{proof}
Insert (25) into (27). The resulting identity is
\[
 \begin{aligned}
 (\tau-\theta)\mathcal B(W)
 &=\sum_rb_rQ_{v_r}(W)+\sum_Hs_H\pi_H(W)\\
 &\quad+\sum_{i=0}^1a_iJ_i(\theta)
              +a_{\mathrm{del}}J_{\mathrm{del}}(\theta),
 \end{aligned}
\]
where
\[
 \mathcal B(W)=1+\sum_j\lambda_jt(B_j,W)
       +\lambda_\triangle B_\triangle
       +\frac{a_0R_0+a_1R_1+a_{\mathrm{del}}D_{\mathrm{del}}}{2p_*}.
\]
The last two contributions to \(\mathcal B\) are nonnegative, so (29)
gives \(\mathcal B>5/8\). Every term on the right is nonnegative at a
global maximum. This proves the bound and, at equality, the vanishing
of each ordinary square, since all \(b_r\) are strictly positive.
\end{proof}

\subsection{Exact spans for equality}

We record three finite consequences of this same certificate. Each is an
explicit rational linear combination, checked coordinate by coordinate;
no numerical rank assertion is used.

\begin{lemma}[Triangle domination]\label{lem:triangle-domination}
In the path type with root edges \(01,12\) and extension vertices \(3,4\),
there is a vector \(G\) in the rational span of its \(53\) strictly
positive square vectors such that
\[
 G(F)=0\quad\text{if \(F\) is triangle-free},\qquad
 G(F)\ge1\quad\text{if \(F\) contains a triangle}.
\]
There are \(24\) coordinates of the first kind and \(48\) of the second.
\end{lemma}
\begin{proof}
The combination is specified in the supplied file
\begin{center}
 \texttt{all\_graph\_path3\_triangle\_domination\_span.json}.
\end{center}
Its checker computes which of the \(72\) flags contain a triangle,
checks that every source vector comes from a strictly positive square
of the final certificate, and adds the rational combination in every
coordinate. The result is zero on all \(24\) triangle-free coordinates
and at least one on all other coordinates. In the supplied witness,
\(G\) is itself a positively weighted source vector: it is the indicator
of the \(48\) triangle-containing flags.
\end{proof}

Once triangle-freeness has been established, only triangle-free flag
coordinates contribute to a conditional flag distribution. The following
two span identities concern the restrictions of the same all-graph square
vectors to those coordinates.

\begin{lemma}[Degree span]\label{lem:degree-span}
In the triangle-free flags of the path type \(01,12\), the vector
\begin{equation}
 g(F)=24\bigl(e_F(1,3)+e_F(1,4)\bigr)-5\tag{30}
\end{equation}
lies in the rational span of the restricted strictly positive square
vectors of that type.
\end{lemma}
\begin{proof}
The witness in
\texttt{all\_graph\_tf\_degree\_codegree\_spans.json} uses \(22\) nonzero
coefficients from the \(53\) path-type sources. The independent checker
identifies the \(24\) triangle-free flags from their edge sets and verifies
(30) in all \(24\) coordinates by rational addition. All source squares
have strictly positive rational weights.
\end{proof}

\begin{lemma}[Codegree span]\label{lem:codegree-span}
Consider the triangle-free flags of the type with the single root edge
\(01\) and isolated root \(2\). Put
\[
 h_i(F)=24\bigl(e_F(i,3)+e_F(i,4)\bigr)-5,\qquad i=0,1.
\]
For each \(i=0,1\), the vector
\begin{equation}
 k_i(F)=5\sum_{j=3}^4e_F(i,j)e_F(2,j)
                 -2\sum_{j=3}^4e_F(2,j)\tag{31}
\end{equation}
lies in the rational span of \(h_0,h_1\) and the restricted strictly
positive square vectors of this type.
\end{lemma}
\begin{proof}
The same witness file gives two explicit combinations, each with \(30\)
nonzero coefficients. Their available sources are the \(56\) positive
edge-type square vectors and the two degree rows \(h_0,h_1\). The checker
computes those rows and the targets (31) from the flag edge sets,
then verifies both identities in all \(32\) triangle-free coordinates.
These are consequences of the single \(1044\)-host certificate; no
separate triangle-free extremal certificate is required.
\end{proof}

\section{The sharp inequality}

\begin{proof}[Proof of the inequalities in Theorems 1.1 and 1.3]
Let \(M\) be the common supremum supplied by Lemma~2.2, and choose a
global graphon maximum attaining it. The Clebsch construction gives
\(M\ge R_*\). Proposition~3.3 shows that its active probability edge
density is strictly greater than \(p_*=5/144\). Lemma~2.5 therefore
produces a global maximum at edge density \(p_*\). All vertex, edge, and
profile conditions used in the exact certificate hold at this graphon.

At fixed density \(p_*\), a ratio at least \(R_*\) is equivalent to
\(\theta\ge\tau\). Corollary~4.2 gives \(\theta\le\tau\), so \(M=R_*\).
By the definition of \(M\), the bound holds for every positive-density
graphon and for every finite nonnegative weighted graph. If the edge mass
is zero, the induced pentagon mass is zero as well. The Clebsch
construction verifies sharpness.
\end{proof}


\section{Eliminating triangles at equality}
\label{sec:triangles}

The equality witness from Lemma 4.3 has a direct consequence for the
underlying graphon. Its use requires only the positive path-type squares and
first-order vertex variation. We work on a probability space in this section.

\begin{proposition}\label{prop:triangle-elimination}
Let $W$ be a global maximizer of $C(W)/E(W)^{5/2}$ among
graphons of positive edge density, where $C(W)=t_{\mathrm{ind}}(C_5,W)/10$
and $E(W)=t(K_2,W)/2$. If
\[
 p(W)=p_\ast:=\frac5{144},\qquad
 t_{\mathrm{ind}}(C_5,W)=\tau:=\frac5{663552},
\]
then $t(K_3,W)=0$.
\end{proposition}

\begin{proof}
Write $d(x)=\int W(x,y)\,dy$ and
$\kappa=5C(W)/(2E(W))=\tau/(2p_\ast)>0$.
Let $L(x)$ denote one half of the rooted induced-pentagon density at $x$.
Bounded signed variations of the vertex measure, followed by normalization
to total mass one, give
\begin{equation}\tag{32}\label{eq:32}
 L(x)=\kappa d(x)\qquad\text{for almost every }x.
\end{equation}
Indeed, the normalization cancels from the homogeneous ratio, and its
derivative in a bounded direction $a$ is a positive constant times
$\int a(x)(L(x)-\kappa d(x))\,dx$.

At the stated values of $p$ and $t$, the verified certificate is an equality.
Every strictly positively weighted square in its path-type block therefore
vanishes. For ordered root locations $x_0,x_1,x_2$, put
\[
 \rho(x_0,x_1,x_2)=W(x_0,x_1)W(x_1,x_2)\bigl(1-W(x_0,x_2)\bigr),
\]
and let $\phi(x_0,x_1,x_2)$ be the conditional vector of the 72 five-vertex
flags of this rooted type, with the root-type factor omitted. The entries of
$\phi$ are nonnegative and sum to one. If $v_j$ is one of the positive square
vectors, then
\[
 \int\rho\,(v_j\cdot\phi)^2=0.
\]
Thus all $v_j\cdot\phi$ vanish for
$\rho(x_0,x_1,x_2)\,dx_0\,dx_1\,dx_2$-almost every root triple.
Lemma 4.3 supplies a rational linear combination $w$ of these vectors which
is zero on the triangle-free flags and at least one on every flag containing
a triangle. Consequently
\[
 0\leq\sum_{F\text{ containing a triangle}}\phi_F\leq w\cdot\phi=0
\]
on the same set of root triples.

Define
\[
 \begin{split}
 T(x)&=\int W(x,v)W(x,w)W(v,w)\,dv\,dw,\\
 r(x)&=\int W(x,y)W(y,z)\bigl(1-W(x,z)\bigr)\,dy\,dz.
 \end{split}
\]
Among the forbidden flag extensions is the event that the endpoint root
$x_0$ and the two extension vertices span a triangle. Its conditional
probability is exactly $T(x_0)$: edges from the extensions to roots 1 and 2
are unrestricted and are summed out. Hence
\begin{equation}\tag{33}\label{eq:33}
 0=\int\rho(x_0,x_1,x_2)T(x_0)\,dx_0\,dx_1\,dx_2
   =\int r(x)T(x)\,dx.
\end{equation}
In the rooted induced cycle $x-a-b-c-d-x$, retain only the factors
$W(x,a)W(a,b)(1-W(x,b))$. All other factors lie in $[0,1]$, and the two
remaining variables integrate to one. The factor $1/2$ in the definition of
$L$ therefore gives
\[
 L(x)\leq\frac{r(x)}2.
\]
Together with (32), this yields $r(x)\geq2\kappa d(x)$. Since every term
is nonnegative, (33) implies $\int d(x)T(x)\,dx=0$. Finally
$0\leq T(x)\leq d(x)^2$. Thus $T$ vanishes both where $d=0$ and where
$d>0$, up to null sets. It follows that $t(K_3,W)=\int T(x)\,dx=0$.
\end{proof}

\section{Classification of equality}
\label{sec:equality}

The preceding triangle-elimination argument reduces equality in the
unrestricted problem to a triangle-free graphon at the reference density.
We now restrict the positive-square vectors of the same certificate to
triangle-free flags and use the two exact conditional spans recorded in
Section 4.

\subsection{Structure at the reference density}

\begin{proposition}\label{prop:reference-equality}
Let $W$ be a triangle-free global ratio maximizer with
$(p,\theta)=(p_\ast,\tau)$. Then $W$ is a balanced Clebsch graphon on an
active set of measure $1/3$, with isolates outside it.
\end{proposition}

\begin{proof}
By Corollary 4.2, every strictly positive square in the exact certificate
vanishes. For each such square vector $v$, formula (26) gives
$v\cdot q(\mathbf x)=0$ for almost every root tuple with respect to its
type measure.

First use the path type $01,12$. The conditional expectation of the vector
in Lemma 4.4 is $48d(x_1)-5$. Its type probability is
\[
 W(x_0,x_1)W(x_1,x_2)(1-W(x_0,x_2)).
\]
The last factor is redundant almost everywhere on the product of the first
two, by triangle-freeness. Integrating over both endpoints therefore gives
\[
 \int d(x)^2(48d(x)-5)^2\,dx=0.
\]
Thus
\begin{equation}\tag{34}\label{eq:34}
 d(x)\in\left\{0,\frac5{48}\right\}\qquad\text{almost everywhere}.
\end{equation}
Let $A=\{x:d(x)>0\}$. Its complement has zero rows almost everywhere, and
\[
 |A|=\frac{p_\ast}{5/48}=\frac13.
\]
Next consider the type consisting of the edge $01$ and isolated root $2$.
Its two edge endpoints have positive degree for almost every root tuple,
so the conditional expectations of $h_0,h_1$ in Lemma 4.5 vanish by (34).
The same lemma then implies
\begin{equation}\tag{35}\label{eq:35}
 c(x_i,x_2)=\frac25d(x_2),\qquad i=0,1,
\end{equation}
almost everywhere with respect to this type measure. Indeed the conditional
expectation of $k_i$ is $10c(x_i,x_2)-4d(x_2)$.

Restrict to $A$ and normalize its measure to one. Call the resulting
graphon $U$. Its degree is $p_0=5/16$. For type tuples entirely in $A$,
(35) gives codegree $\mu=1/8$ in the new normalization. We now show that
these relations determine $U$.

Write
\[
 c_U(x,z)=\int U(x,y)U(z,y)\,dy.
\]
Triangle-freeness gives $U(x,z)c_U(x,z)=0$ almost everywhere. Integrate
the square of $c_U(x,z)-\mu$ against the edge-plus-isolate type probability
and integrate out its middle root. Since
\[
 \int U(x,y)(1-U(y,z))\,dy=p_0-c_U(x,z),
\]
we obtain
\begin{equation}\tag{36}\label{eq:36}
 (1-U(x,z))(p_0-c_U(x,z))(c_U(x,z)-\mu)^2=0
 \quad\text{almost everywhere}.
\end{equation}
If $0<U(x,z)<1$, triangle-freeness forces $c_U(x,z)=0$, contradicting
(36). Hence $U$ is $\{0,1\}$-valued. Edges have codegree zero, and
nonedges have codegree $\mu$ or $p_0$.

Regard $f_x=U(x,\cdot)$ as an element of $L^2([0,1])$. After a null set
is discarded, all rows are $\{0,1\}$-valued and have integral $p_0$.
Therefore
\[
 \|f_x-f_z\|_2^2=2p_0-2c_U(x,z).
\]
In particular $c_U(x,z)=p_0$ precisely when the two rows agree in $L^2$.
Each twin class $T_x=\{z:f_z=f_x\}$ is measurable: it is the zero set of
the measurable function $z\mapsto\int|U(z,y)-U(x,y)|^2\,dy$.

Let $a(x)=|T_x|$. Almost every twin is a nonneighbor, and every other
nonneighbor has codegree $\mu$. Regularity gives
$\int c_U(x,z)\,dz=p_0^2$, whence
\[
 p_0^2=a(x)p_0+(1-p_0-a(x))\mu.
\]
Substituting $p_0=5/16$ and $\mu=1/8$ gives $a(x)=1/16$ for almost every
$x$. Choose generic representatives successively outside the classes already
selected. Since each selected class has measure $1/16$, exactly sixteen
disjoint measurable classes cover the space up to null sets.

The kernel is constant on every class rectangle. To see this, on a twin
class $A_i$ all rows agree in $L^2$ with one function $f_i$. Fubini and
symmetry give $U(x,z)=f_i(z)=f_j(x)$ almost everywhere on $A_i\times A_j$.
Integrating over either coordinate makes this value constant. Every block
value is zero or one, and diagonal blocks are zero because their pairs have
positive codegree. The quotient is consequently a simple graph on sixteen
vertices, of degree five, with zero common neighbors for adjacent vertices
and two for distinct nonadjacent vertices.

These parameters determine the Clebsch graph elementarily. Fix a vertex
$v$ with independent neighbors $u_1,\ldots,u_5$. Every other vertex has two
neighbors among these five. Each pair $u_i,u_j$ has exactly two common
neighbors, one being $v$; label the other $w_{ij}$. This labels the remaining
ten vertices by two-element subsets of $\{1,\ldots,5\}$. Two pair vertices
sharing a label cannot be adjacent, since that would create a triangle. A
pair vertex needs three further neighbors after $u_i,u_j$, and there are
exactly three disjoint pairs. It joins precisely those three.

Choose the five Clebsch generators $s_1,\ldots,s_5$ from the introduction.
The assignment
\[
 v\mapsto0,\qquad u_i\mapsto s_i,\qquad w_{ij}\mapsto s_i+s_j
\]
is an isomorphism to the Clebsch graph: two pair sums are adjacent exactly
when their pairs are disjoint. Thus $U$ is balanced Clebsch, proving the
proposition.
\end{proof}

\subsection{Arbitrary equality graphons}

\begin{proof}[Equality in Theorem 1.3]
Every member of $\mathcal E_{\mathrm{Cl}}$ attains the bound by the Clebsch
counts and homogeneity. Conversely, suppose $R(W)=R_\ast$ and $E(W)>0$.
The sharp inequality has already been proved for all graphons, so $W$ is
a global maximum over all positive edge densities. Remove its actual
zero-degree set and normalize the remaining vertex measure to one; the
ratio and global maximality are unchanged.

Proposition 3.3 applies to this graphon even if its values are fractional.
Its active density is therefore greater than $p_\ast$. Pad by isolates as
in Lemma 2.5. The padded graphon is still a global maximum, has
$p=p_\ast$ and $\theta=\tau$, and satisfies the vertex first-order
conditions. Proposition 6.1 makes it triangle-free. Proposition 7.1 then
identifies its active part as a balanced Clebsch graphon.

Removing the added isolates and undoing the common vertex-measure scaling
proves $W\in\mathcal E_{\mathrm{Cl}}$. This argument classifies an arbitrary
equality graphon. It does not infer graphon equality from equality on finite
graphs, and it does not impose edge optimality on a merely template-optimal
finite weighting.
\end{proof}

\begin{proof}[Equality in Theorem 1.1]
Normalize the total positive vertex weight to one and represent the weighted
graph as a step graphon with those class measures. At equality, Theorem 1.3
gives a balanced Clebsch active part and isolates. Each original
positive-weight vertex has a constant neighborhood row, so its cell belongs
to one of the twin classes, or to the isolate set, up to null sets. The
sixteen active classes therefore have equal total vertex weight, and their
adjacencies are exactly the Clebsch adjacencies. The converse follows from
the counts $E=40a^2$ and $C=192a^5$ when each active class has total weight
$a$.
\end{proof}

\section{Stability on the edge scale}
\label{sec:stability}

The number of isolated or low-degree vertices is not controlled by an edge
constraint. We first remove a small amount of edge mass to obtain a core
with positive normalized minimum degree. Compactness then places that core
in the neighborhood covered by the local stability theorem. A final
rounding argument produces equal integer class sizes without an additive
error term.

\begin{theorem}[Edge-scaled stability]\label{thm:edge-stability}
There are absolute constants $K<\infty$ and $\varepsilon_0>0$ such that the
following holds. Let $G$ be a finite graph with $m>0$ edges, and suppose that
$0\leq\varepsilon\leq\varepsilon_0$ and
\[
 I_5(G)\geq(R_\ast-\varepsilon)\left(\frac m5\right)^{5/2}.
\]
There is a graph $B$ on the same vertex set, consisting of a Clebsch blow-up
with sixteen classes of the same positive integer size and isolated
vertices, for which
\[
 |E(G)\mathbin{\triangle}E(B)|\leq K\sqrt\varepsilon\,m.
\]
If $s$ is the common class size, then
\[
 |m-40s^2|\leq K\sqrt\varepsilon\,m,\qquad
 s=(1+O(\sqrt\varepsilon))\sqrt{\frac m{40}}.
\]
\end{theorem}

There is no restriction on $|V(G)|$ and no additional rounding term. The
square-root dependence on the deficit is optimal, as shown in Proposition
8.5 below.

\subsection{Removing low-degree vertices}

\begin{lemma}[Weighted peeling]\label{lem:peeling}
Let $(H,x)$ be a finite weighted graph with $E(H,x)=1$ and
$C(H,x)\geq c_\ast-\eta$, where $0\leq\eta\leq c_\ast/4$. Repeatedly remove
a vertex whose current weighted degree is at most $c_\ast$, keeping the
other weights fixed. Write $(H',x')$ for the remaining core,
$E'=E(H',x')$, $C'=C(H',x')$, and
$W'=\sum_{v\in V(H')}x'_v$. Then
\begin{align}
 1-E'&\leq\frac{4\eta}{3c_\ast},\tag{37}\label{eq:37}\\
 \frac{C'}{(E')^{5/2}}&\geq C(H,x),\tag{38}\label{eq:38}\\
 W'&<\frac2{c_\ast}.\tag{39}\label{eq:39}
\end{align}
In particular $E'>0$. After the core weights are normalized to have sum one,
its weighted degrees are all at least $c_\ast^2/2$.
\end{lemma}

\begin{proof}
At an intermediate stage, let $E_t$ be the current edge mass and let $d_i$
be the current weighted degree of a vertex $i$ about to be removed. Fixing
the edge opposite $i$ in a pentagon gives
\[
 C_i\leq\frac14 E_t d_i^2.
\]
Indeed, for a fixed opposite edge $uv$, the two neighbors of $i$ lie in the
exclusive $u$-neighbor and exclusive $v$-neighbor subsets of $N(i)$.
Inducedness makes these sets disjoint even when the graph has triangles.
Their weight product is at most $d_i^2/4$. Removing $i$ destroys edge mass
$x_i d_i$ and pentagon mass $x_i C_i$. As $E_t\leq1$ and $d_i\leq c_\ast$,
the latter is at most $(c_\ast/4)x_i d_i$.

Put $\alpha=1-E'$. Each removed edge and each destroyed pentagon is charged
at its first deleted endpoint, so
\begin{equation}\tag{40}\label{eq:40}
 0\leq C(H,x)-C'\leq\frac{c_\ast}4\alpha.
\end{equation}
Theorem 1.1, applied to the surviving core, gives
\[
 c_\ast-\eta-\frac{c_\ast}4\alpha\leq C'
 \leq c_\ast(1-\alpha)^{5/2}.
\]
Since $1-(1-\alpha)^{5/2}\geq\alpha$ for $0\leq\alpha\leq1$, this proves
(37), and hence $E'\geq2/3$. Moreover $C(H,x)\geq3c_\ast/4$, so
\[
 C'\geq C(H,x)-\frac{c_\ast}4\alpha
 \geq C(H,x)(1-\alpha)\geq C(H,x)(1-\alpha)^{5/2}.
\]
This is (38).

Every surviving vertex has weighted degree greater than $c_\ast$.
Consequently
\[
 c_\ast W'<\sum_{i\in V(H')}x_i d_i=2E'\leq2.
\]
After dividing the vertex weights by $W'$, every weighted degree exceeds
\[
 \frac{c_\ast}{W'}>\frac{c_\ast^2}{2E'}\geq\frac{c_\ast^2}2.
\]
The degrees must be recomputed after each removal; this is why the procedure
is iterative.
\end{proof}

\subsection{Entering the local neighborhood}
Let $W_{\mathrm{Cl}}$ denote a full balanced Clebsch graphon. We use
$\delta_\square$ for cut distance and $\delta_1$ for $L^1$ distance, in each
case taking the infimum over measure-preserving relabelings. The standard
compactness of graphons in cut distance and continuity of fixed finite
homomorphism densities will be used only qualitatively.

\begin{lemma}\label{lem:positive-degree-compactness}
Fix $\mu>0$. If $W_j$ are graphons with $d_{W_j}\geq\mu$ almost everywhere
and $R(W_j)\to R_\ast$, then $\delta_1(W_j,W_{\mathrm{Cl}})\to0$.
\end{lemma}

\begin{proof}
Their edge densities are at least $\mu$. From any subsequence, compactness
extracts a further subsequence converging in cut distance to a graphon $W$.
Choose relabelings giving cut-norm convergence. For every measurable
$S\subseteq[0,1]$,
\[
 \int_S d_W=\lim_j\int_S d_{W_j}\geq\mu|S|,
\]
because the difference of the two integrals is bounded by the cut norm.
Thus $d_W\geq\mu$ almost everywhere. The equality $R(W)=R_\ast$ passes to
the limit by induced-density continuity; the denominator in $R$ stays
bounded away from zero. Theorem 1.3 identifies $W$ as balanced Clebsch on
its active set, with isolates outside. The minimum-degree condition
excludes a positive-measure isolate set, so $W$ is a full balanced Clebsch
graphon.

This proves convergence of the whole sequence in cut distance. To obtain
$L^1$ convergence, fix the sixteen classes $A_1,\ldots,A_{16}$ of a labeling
of $W_{\mathrm{Cl}}$. Since $W_{\mathrm{Cl}}$ is $\{0,1\}$-valued,
$U-W_{\mathrm{Cl}}$ has a constant sign on each $A_i\times A_j$. Hence
\[
 \|U-W_{\mathrm{Cl}}\|_1
 =\sum_{i,j}\left|\int_{A_i\times A_j}(U-W_{\mathrm{Cl}})\right|
 \leq256\|U-W_{\mathrm{Cl}}\|_\square.
\]
Taking infima over relabelings completes the proof.
\end{proof}

\subsection{Rounding whole vertices}
Write $\mathcal E_{\mathrm{Cl}}$ for the family defined in the introduction.
The next lemma is useful because balancing real class masses by rounding
their sizes separately would introduce an unnecessary additive error.

\begin{lemma}[Integral rounding]\label{lem:integral-rounding}
Let $W_G$ be the probability graphon of an $n$-vertex graph, and let
$U\in\mathcal E_{\mathrm{Cl}}$ have active mass $q\geq1/2$. There is a graph
$B$ on $V(G)$, consisting of sixteen equal integer Clebsch classes and
isolates, such that
\[
 \|W_G-W_B\|_1\leq897\|W_G-U\|_1.
\]
The common class size is positive when $\|W_G-U\|_1$ is sufficiently small.
\end{lemma}

\begin{proof}
Put $\zeta=\|W_G-U\|_1$. Let $B_0$ be the isolate set of $U$, and let
$B_1,\ldots,B_{16}$ be its active classes, each of measure $b=q/16$.
Denote the corresponding row functions by $r_0=0,r_1,\ldots,r_{16}$.
Different rows have distance at least
\begin{equation}\tag{41}\label{eq:41}
 \rho=\frac{5q}{16}\geq\frac5{32}.
\end{equation}
A nonzero row meets five classes, while distinct Clebsch rows differ on six
or ten classes.

For a vertex $v$, let $I_v$ be its cell in $W_G$, of measure $1/n$, and let
$g_v$ be its constant row. Set
\[
 p_{v,i}=n|I_v\cap B_i|,\qquad
 e_{v,i}=\|g_v-r_i\|_1,\qquad
 e_v=\sum_{i=0}^{16}p_{v,i}e_{v,i}.
\]
Assign the entire vertex $v$ to a type $i(v)$ minimizing $e_{v,i}$. If this
minimum is at least $\rho/2$, then $e_v\geq\rho/2$. Otherwise, the triangle
inequality and (41) give $e_{v,j}\geq\rho/2$ for every $j\ne i(v)$. In both
cases
\[
 1-p_{v,i(v)}\leq\frac{2e_v}{\rho}.
\]
The total measure whose type changes is therefore
\begin{equation}\tag{42}\label{eq:42}
 M=\frac1n\sum_v(1-p_{v,i(v)})\leq\frac{2\zeta}{\rho}.
\end{equation}
Let $T$ be the Clebsch step graphon given by these whole-vertex assignments,
and let $s_i$ be the mass of its $i$th active class. Changing types on mass
$M$ gives
\[
 \|T-U\|_1\leq2M,\qquad\sum_{i=1}^{16}|s_i-b|\leq2M.
\]
Set $s_{\min}=\min_i s_i$. Remove excess vertices from each active class
until every class has size $ns_{\min}$, assigning the removed vertices to
the isolate type. This uses whole vertices because all $ns_i$ are integers.
The total mass reassigned is
\[
 \Delta=\sum_{i=1}^{16}(s_i-s_{\min})
 \leq17\sum_{i=1}^{16}|s_i-b|\leq34M.
\]
For the first inequality, if $b\geq s_{\min}$ use
$b-s_{\min}\leq\sum_i|s_i-b|$; if $b<s_{\min}$, the inequality is immediate.
The additional $L^1$ cost is at most $2\Delta\leq68M$. The resulting graph
$B$ is exactly balanced, and
\[
 \|W_G-W_B\|_1\leq\zeta+2M+68M
 \leq\left(1+\frac{140}{\rho}\right)\zeta\leq897\zeta.
\]
For sufficiently small $\zeta$, (42) and $b\geq1/32$ also imply $s_i>0$
for every $i$.
\end{proof}

This proof applies after arbitrary relabelings: the cells $I_v$ need not be
intervals. It also handles a nonattained distance infimum. Apply the lemma
to a sequence of comparison graphons approaching the infimum and use the
finiteness of the possible edit distances on a fixed vertex set. If the
infimum is zero, taking the error small enough to permit fewer than one
edge edit forces an exact copy.

\subsection{Proof of the stability theorem}

\begin{proof}[Proof of Theorem 8.1]
Give every vertex of $G$ weight $1/\sqrt m$. Then $E=1$ and
\[
 C\geq c_\ast-\eta,\qquad\eta=\frac{\varepsilon}{5^{5/2}}.
\]
Apply Lemma 8.2. The peeling removes at most
\begin{equation}\tag{43}\label{eq:43}
 \frac{4\varepsilon}{3c_\ast5^{5/2}}m
\end{equation}
edges and leaves a graph $H$ with
\[
 n_H<\frac{2\sqrt m}{c_\ast}.
\]
Its probability graphon has minimum degree at least $c_\ast^2/2$ and ratio
at least $R_\ast-\varepsilon$.

By Lemma 8.3, there is an absolute positive threshold for $\varepsilon$
below which $W_H$, after a relabeling, lies in the neighborhood of Theorem
A.1. Let $c_{\mathrm{loc}}>0$ be the constant in that theorem. Its
induced-pentagon functional is exactly the functional used here, so the
local conclusion is
\[
 d_1(W_H,\mathcal E_{\mathrm{Cl}})\leq
 \sqrt{\frac{\varepsilon}{c_{\mathrm{loc}}}}.
\]
Comparison graphons sufficiently close to this infimum have active mass
at least $1/2$ when $\varepsilon$ is small. To see this, also choose the
compactness neighborhood small enough that $W_H$ is close to the full
Clebsch graphon in $L^1$. Edge-density continuity and
$t(K_2,U)=(5/16)q^2$ then force $q$ close to one.

Lemma 8.4 gives a balanced Clebsch blow-up $B_H$ with isolates on $V(H)$.
Taking the comparison distance to its infimum, the number of edits inside
$H$ is at most
\[
 \frac{897n_H^2}2\sqrt{\frac{\varepsilon}{c_{\mathrm{loc}}}}
 \leq\frac{1794}{c_\ast^2\sqrt{c_{\mathrm{loc}}}}
 \sqrt\varepsilon\,m.
\]
Extend $B_H$ to the original vertex set by making every peeled vertex
isolated. Adding (43), and using $\varepsilon\leq\sqrt\varepsilon$ for
$\varepsilon\leq1$, proves the edit bound. For example, once
$\varepsilon_0$ has been chosen small enough for the preceding neighborhood
and nonemptiness requirements, one may take
\[
 K=\frac{1794}{c_\ast^2\sqrt{c_{\mathrm{loc}}}}
   +\frac4{3c_\ast5^{5/2}}.
\]
The value of $K$ is not optimized. Finally, a balanced Clebsch blow-up with
common class size $s$ has $40s^2$ edges. Its edge count differs from $m$ by
at most the number of edits, proving both remaining conclusions.
\end{proof}

\subsection{Optimality of the exponent}

\begin{proposition}\label{prop:sqrt-optimality}
The factor $\sqrt\varepsilon$ in Theorem 8.1 cannot be replaced by
$o(\sqrt\varepsilon)$, uniformly over finite graphs.
\end{proposition}

\begin{proof}
Choose two nonadjacent vertices $u,v$ of $\mathrm{Cl}$. For integers
$k\geq2$ and $1\leq r<k$, form a Clebsch blow-up $G_{k,r}$ with class sizes
$k+r$ at $u$, $k-r$ at $v$, and $k$ at the other fourteen vertices.

Each Clebsch vertex lies in sixty pentagons, and each nonadjacent pair lies
in twelve. For the latter count, the pair has two common neighbors and six
three-edge paths between its endpoints. Each choice of a common neighbor
and such a path forms a pentagon, by triangle-freeness. The six paths can
also be counted from the five neighbors of one endpoint: two are adjacent
to the other endpoint and contribute zero codegree, while the other three
contribute codegree two.

The edge and pentagon polynomials are affine in each class size. The linear
changes at $u$ and $v$ cancel, and there is no quadratic edge term because
$u,v$ are nonadjacent. Thus, exactly,
\[
 |V(G_{k,r})|=16k,\qquad m=40k^2,\qquad
 I_5(G_{k,r})=192k^5-12k^3r^2.
\]
Its normalized deficit is
\[
 \varepsilon_{k,r}=R_\ast-\frac{I_5(G_{k,r})}{(m/5)^{5/2}}
 =\frac{R_\ast r^2}{16k^2}.
\]
The pair $u,v$ has two common neighbors. Its three exclusive neighbors on
each side have unchanged class size $k$, so $G_{k,r}$ has $3k$ vertices of
degree $5k+r$, $3k$ of degree $5k-r$, and $10k$ of degree $5k$.

A balanced Clebsch blow-up with isolates on the same $16k$ vertices has
common class size at most $k$, and hence maximum degree at most $5k$.
Each of the $3k$ high-degree vertices must lose at least $r$ in degree. One
edge edit changes the total absolute degree discrepancy by at most two.
Therefore every such target $B$ satisfies
\[
 |E(G_{k,r})\mathbin{\triangle}E(B)|\geq\frac32kr
 =\frac3{20\sqrt{R_\ast}}\sqrt{\varepsilon_{k,r}}\,m.
\]
Taking $r=1$ and $k\to\infty$ proves the assertion.
\end{proof}

\section{Linear structural stability and quadratic mass stability}
\label{sec:linear-stability}

The local repair argument proves a stronger conclusion than its quadratic
distance estimate alone. Edge errors have a linear cost in the extremal
deficit; only the imbalance of the sixteen class masses has a quadratic
cost. We retain these two conclusions separately.

Let $\mathcal T_{\mathrm{Cl}}$ denote the class of graphons obtained from an
arbitrary sixteen-class Clebsch blow-up and an isolated class. The sixteen
active masses need not be equal. For a presented member $T$ with active
masses $s_1,\ldots,s_{16}$, put
\[
 q(T)=\sum_{i=1}^{16}s_i,\qquad
 V(T)^2=\sum_{i=1}^{16}(s_i-q(T)/16)^2.
\]

\begin{proposition}[The two local defects]\label{prop:two-defects}
There are constants $A,B>0$ and $\delta>0$ such that, whenever
$\|W-W_{\mathrm{Cl}}\|_1<\delta$, a graphon $T\in\mathcal T_{\mathrm{Cl}}$
exists with all active masses at least $1/32$ and
\begin{equation}\tag{44}\label{eq:44}
 \|W-T\|_1\leq AD(W),\qquad V(T)^2\leq BD(W),\qquad
 D(W)=R_\ast-R(W).
\end{equation}
\end{proposition}

\begin{proof}
Use the sets $B_0,Z,M_1,\ldots,M_{16}$ constructed in the proof of Theorem
A.1; here $B_0$ denotes its bad-row set to avoid confusion with the constant
in the present statement. Write $F=C/E^{5/2}=R/5^{5/2}$ as in that proof.
Take $T$ to be the repaired Clebsch step graphon on $M=\bigcup_i M_i$, and
put isolates outside $M$. Before balancing the masses, the four-defect
inequality is
\[
 F(W_{\mathrm{Cl}})-F(W)
 \geq c_B\mu_{B_0}+c_ZD_Z+c_Ee_M+c_HV(T)^2.
\]
The three operations that produce $T$ cost at most
$2\mu_{B_0}+2D_Z+e_M$ in $L^1$. Consequently one may take
\[
 A=\frac1{5^{5/2}\min\{c_B/2,c_Z/2,c_E\}},\qquad
 B=\frac1{5^{5/2}c_H}.
\]
The mass-closeness estimate in that proof gives
$s_i=1/16+O(\|W-W_{\mathrm{Cl}}\|_1)$ uniformly in $i$. Decreasing the
fixed neighborhood ensures $s_i\geq1/32$.
\end{proof}

\begin{lemma}[Rounding without balancing]\label{lem:unbalanced-rounding}
Let $W_G$ be the probability graphon of a graph on $n$ vertices, and let
$U\in\mathcal T_{\mathrm{Cl}}$ have active masses $s_i\geq1/32$. Set
$\zeta=\|W_G-U\|_1$ and $\rho=5/32$. There is a Clebsch blow-up $H$, with
isolates, on the same vertex set, such that
\begin{align}
 \|W_G-W_H\|_1&\leq(1+4/\rho)\zeta,\tag{45}\label{eq:45}\\
 V(W_H)^2&\leq2V(U)^2+32\zeta^2/\rho^2.\tag{46}\label{eq:46}
\end{align}
All sixteen active classes of $H$ are nonempty if $\zeta<\rho/128$.
\end{lemma}

\begin{proof}
The seventeen row prototypes of $U$ are pairwise separated in $L^1$ by at
least $\rho$. Indeed, a nonzero Clebsch row meets five classes, while distinct
nonzero rows differ on at least six classes.

For each vertex cell $I_v$, write $p_{v,i}=n|I_v\cap A_i|$ and
$e_{v,i}=\|W_G(v,\cdot)-r_i\|_1$, including the isolated type $i=0$.
Assign the whole vertex to a type minimizing $e_{v,i}$. The argument of
Lemma 8.4 gives total changed type mass
\[
 M=\frac1n\sum_v(1-p_{v,i(v)})\leq\frac{2\zeta}{\rho}.
\]
The induced whole-vertex template $W_H$ differs from $U$ only on pairs
with at least one endpoint in this changed set. Therefore
$\|U-W_H\|_1\leq2M$, proving (45). No trimming is performed.

Let $s'$ be the vector of the sixteen rounded active masses. The changed
mass estimate implies $\|s'-s\|_1\leq2M$. Let $P$ be orthogonal projection
onto the sum-zero subspace of $\mathbb R^{16}$. Then
\[
 \|Ps'\|_2^2\leq2\|Ps\|_2^2+2\|P(s'-s)\|_2^2
 \leq2V(U)^2+8M^2,
\]
which is (46). Finally
$s'_i\geq s_i-2M\geq1/32-4\zeta/\rho>0$ under the stated smallness
assumption. Each $ns'_i$ is an integer because only whole vertices were
assigned. The vertex cells can be arbitrary measurable sets, so the argument
also applies after a graphon relabeling.
\end{proof}

\begin{theorem}[Linear structural stability]\label{thm:linear-stability}
There are absolute constants $K<\infty$ and $\varepsilon_0>0$ such that the
following holds. Suppose $G$ has $m>0$ edges and
\[
 I_5(G)\geq(R_\ast-\varepsilon)(m/5)^{5/2},\qquad
 0\leq\varepsilon\leq\varepsilon_0.
\]
Then a graph $H$ on the same vertex set, consisting of a Clebsch blow-up with
positive integer class sizes $n_1,\ldots,n_{16}$ and isolated vertices,
satisfies
\begin{equation}\tag{47}\label{eq:47}
 |E(G)\mathbin{\triangle}E(H)|+\sum_{i=1}^{16}(n_i-N/16)^2
 \leq K\varepsilon m,\qquad N=\sum_{i=1}^{16}n_i.
\end{equation}
In particular,
\begin{equation}\tag{48}\label{eq:48}
 \left|m-\frac{5N^2}{32}\right|\leq K'\varepsilon m,\qquad
 N=\sqrt{\frac{32m}5}+O(\varepsilon\sqrt m),
\end{equation}
and every class satisfies
\[
 n_i=(1+O(\sqrt\varepsilon))\sqrt{m/40}.
\]
All constants are absolute. There is no additive rounding term.
\end{theorem}

\begin{proof}
Give every vertex weight $1/\sqrt m$. The weighted peeling lemma removes
at most $4\varepsilon m/(3R_\ast)$ edges, and leaves a core $G'$ with
$n<2\sqrt m/c_\ast$ vertices, normalized minimum degree at least
$c_\ast^2/2$, and ratio at least $R_\ast-\varepsilon$. By the positive-degree
compactness lemma, after relabeling its probability graphon lies in the
neighborhood of Proposition 9.1 when $\varepsilon_0$ is small enough.
Obtain the template $U$ there. It satisfies
\[
 \zeta\leq A\varepsilon,\qquad V(U)^2\leq B\varepsilon.
\]
Decrease $\varepsilon_0$ further so that $A\varepsilon_0<\rho/128$ and
$\varepsilon_0\leq1$. Lemma 9.2 gives whole-vertex classes with
\[
 \begin{split}
 \|W_{G'}-W_H\|_1&\leq(1+4/\rho)A\varepsilon,\\
 V(W_H)^2&\leq B'\varepsilon,\qquad B'=2B+32A^2/\rho^2.
 \end{split}
\]
The number of edge edits inside the core is $n^2/2$ times its $L^1$
distance. Also $\sum_i(n_i-N/16)^2=n^2V(W_H)^2$. Using
$n^2<4m/c_\ast^2$ and putting all peeled vertices into the isolated class
proves the two estimates separately with constants
\[
 K_E=\frac4{3R_\ast}+\frac{2(1+4/\rho)A}{c_\ast^2},\qquad
 K_V=\frac{4B'}{c_\ast^2}.
\]
Take $K=K_E+K_V$ in (47).

For completeness, let $A_{\mathrm{Cl}}$ be the Clebsch adjacency matrix
and $h_i=n_i-N/16$. Since the graph is five-regular and $\sum_i h_i=0$,
\[
 e(H)=40(N/16)^2+\frac12 h^TA_{\mathrm{Cl}}h
 =\frac{5N^2}{32}+\frac12 h^TA_{\mathrm{Cl}}h.
\]
The elementary bound $|h^TA_{\mathrm{Cl}}h|\leq5\|h\|_2^2$ follows from
$2|h_i h_j|\leq h_i^2+h_j^2$ summed over the edges. Hence one may take
$K'=K_E+(5/2)K_V$ in the first estimate of (48). For smaller
$\varepsilon_0$, this places $N$ between fixed positive multiples of
$\sqrt m$. Factoring the difference of squares yields the second estimate.
Finally $|h_i|\leq\|h\|_2=O(\sqrt{\varepsilon m})$, which gives the
class-size conclusion.
\end{proof}

\begin{proposition}[Optimality of linear structural stability]
\label{prop:linear-optimality}
The factor $\varepsilon$ in the edit-distance conclusion of Theorem 9.3
cannot be replaced uniformly by $o(\varepsilon)$.
\end{proposition}

\begin{proof}
Start with a balanced Clebsch blow-up of common class size $k\geq2$, and
add a single edge between two vertices in one class. The two endpoints are
now true twins. An induced pentagon cannot contain both, and all other
induced pentagons are unchanged. Thus the new graph has
\[
 m_k=40k^2+1,\qquad I_5=192k^5,\qquad
 \varepsilon_k=R_\ast\left[1-\left(\frac{40k^2}{40k^2+1}\right)^{5/2}\right].
\]
Its distance to the family of Clebsch blow-ups and isolates is exactly one:
deleting the added edge suffices, and the graph contains a triangle whereas
every graph in that family is triangle-free. Furthermore $\varepsilon_k\to0$
and $\varepsilon_km_k\to(5/2)R_\ast$. A bound of the form
$o(\varepsilon_k)m_k$ would tend to zero, contradicting this integer
distance of one.
\end{proof}

\subsection{A two-sided geometric description of the deficit}
For a graph $G$ with $m>0$ edges, write
\[
 \Delta(G)=c_\ast m^{5/2}-I_5(G).
\]
Let $\mathcal B_m(G)$ consist of the Clebsch blow-ups with isolates on
$V(G)$ whose sixteen class sizes obey
\[
 \frac12\sqrt{m/40}\leq n_i\leq2\sqrt{m/40}.
\]
For such a template put $N=\sum_i n_i$ and define
\[
 J_G(B)=|E(G)\mathbin{\triangle}E(B)|+\sum_i(n_i-N/16)^2.
\]

\begin{proposition}[Two-sided structural energy]\label{prop:structural-energy}
For absolute constants $c,C>0$, every graph sufficiently close to the sharp
ratio satisfies
\begin{equation}\tag{49}\label{eq:49}
 c\,m^{3/2}\min_{B\in\mathcal B_m(G)}J_G(B)
 \leq\Delta(G)\leq
 C\,m^{3/2}\min_{B\in\mathcal B_m(G)}J_G(B).
\end{equation}
In particular $\mathcal B_m(G)$ is nonempty in this range. The upper
estimate holds whenever $\mathcal B_m(G)$ is nonempty, without a separate
near-extremality assumption.
\end{proposition}

\begin{proof}
For a sufficiently small fixed normalized deficit, the class-size conclusion
of Theorem 9.3 places its template in $\mathcal B_m(G)$. Substituting
$\varepsilon=5^{5/2}\Delta(G)/m^{5/2}$ in that theorem proves the lower
estimate in (49).

For the other direction, fix $B\in\mathcal B_m(G)$ and put
\[
 t=|E(G)\mathbin{\triangle}E(B)|,\qquad
 V=\sum_i(n_i-N/16)^2,\qquad m_B=e(B).
\]
An original induced pentagon of $B$ is still induced in $G$ unless one of
its ten pairs is modified. Each modified pair belongs to at most
$\binom{N-2}3\leq N^3/6$ original five-sets; pairs meeting an isolated
vertex of $B$ belong to none. Thus
\begin{equation}\tag{50}\label{eq:50}
 I_5(G)\geq I_5(B)-tN^3/6.
\end{equation}
This estimate depends on the number of active template vertices, not on the
possibly much larger number of vertices of $G$.

Consider the weighted-template defect
\[
 Q(x)=c_\ast E(x)^{5/2}-P(x).
\]
At every positive constant vector $b\mathbf1$ we have $Q=0$ and
$\nabla Q=0$. The latter follows either by differentiating the sharp equality
along class weights, or directly from $E(\mathbf1)=40$, $P(\mathbf1)=192$,
$\nabla E(\mathbf1)=5\mathbf1$ and $\nabla P(\mathbf1)=60\mathbf1$.
Set $b=N/16$ and $h_i=n_i-b$. Every point of the line segment
$b\mathbf1+uh$, $0\leq u\leq1$, has coordinates between fixed positive
multiples of $\sqrt m$. On this segment the Hessian of $Q$ is
$O(m^{3/2})$: the polynomial $P$ has degree five, and the same scaling
follows directly by differentiating $E^{5/2}$, with $E$ bounded above and
below by fixed positive multiples of $m$. Taylor's formula therefore gives
\[
 0\leq c_\ast m_B^{5/2}-I_5(B)=Q(b\mathbf1+h)
 \leq C_0m^{3/2}\|h\|_2^2=C_0m^{3/2}V.
\]
Also $m_B$ and $m$ are within fixed positive constant factors, and
$|m_B-m|\leq t$, so
\[
 |m^{5/2}-m_B^{5/2}|\leq C_1m^{3/2}t.
\]
Since $N=O(\sqrt m)$, combine these estimates with (50) to obtain
$\Delta(G)\leq Cm^{3/2}(t+V)$. Taking the minimum completes the proof.
There are finitely many vertex partitions, so the minimum is attained
whenever the family is nonempty.
\end{proof}

\section{The arithmetic size of the finite-edge deficit}
\label{sec:arithmetic-size}

The error term in Corollary 1.2 cannot be improved uniformly to $o(m^2)$.
More precisely, its order is governed by the distance of an effective vertex
count from the integers. This conclusion uses the linear stability theorem
for unbalanced Clebsch blow-ups, including its squared class-size estimate.

For an integer $N\geq0$, write $N=16k+r$, with $0\leq r<16$, and put
\[
 v(N)=\frac{r(16-r)}{16},\qquad
 \Psi(m)=\min_{N\in\mathbb Z_{\geq0}}
 \left(\left|m-\frac5{32}N^2\right|+v(N)\right).
\]
The minimum is attained, and $0\leq v(N)\leq4$. As before, let
$\rho(C_5,m)$ denote the maximum number of induced pentagons among graphs
with exactly $m$ edges, and set
\[
 \Delta(m)=c_\ast m^{5/2}-\rho(C_5,m).
\]

\begin{theorem}[Arithmetic deficit law]\label{thm:arithmetic-deficit}
There exist absolute constants $a_0,b_0>0$ and an integer $m_0$ such that,
for every integer $m\geq m_0$,
\begin{equation}\tag{51}\label{eq:51}
 a_0m^{3/2}\Psi(m)\leq\Delta(m)\leq b_0m^{3/2}\Psi(m).
\end{equation}
In particular, $\Psi(m)=0$ exactly when $m=40k^2$ for an integer $k\geq0$.
For all other sufficiently large integers $m$,
\begin{equation}\tag{52}\label{eq:52}
 \Delta(m)\asymp m^2\operatorname{dist}
 \left(\sqrt{\frac{32m}5},\mathbb Z\right)+m^{3/2},
\end{equation}
where the two comparison constants are absolute.
\end{theorem}

\begin{proof}
We first record the role of the integer class sizes. If
$n_1,\ldots,n_{16}$ are nonnegative integers with sum $N$, then
\begin{equation}\tag{53}\label{eq:53}
 V:=\sum_{i=1}^{16}(n_i-N/16)^2\geq v(N).
\end{equation}
Indeed, whenever two entries differ by at least two, moving one unit from
the larger to the smaller decreases the sum of their squares. The minimum
is therefore attained by $r$ entries equal to $k+1$ and $16-r$ entries equal
to $k$; its value is $v(N)$.

Let $A$ be the adjacency matrix of $\mathrm{Cl}$, and write
$n=(N/16)\mathbf1+h$, where $\sum_i h_i=0$. Since $A\mathbf1=5\mathbf1$,
the edge count of this blow-up is
\begin{equation}\tag{54}\label{eq:54}
 E(n)=\frac12n^TAn=\frac5{32}N^2+\frac12h^TAh.
\end{equation}
The operator norm of $A$ is at most five. Consequently
\begin{equation}\tag{55}\label{eq:55}
 \left|E(n)-\frac5{32}N^2\right|\leq\frac52 V.
\end{equation}
For the lower bound in (51), choose an extremal graph $G$ with $m$ edges.
Such a graph exists, since discarding isolates leaves at most $2m$ vertices.
The earlier lower construction gives $\Delta(m)=O(m^2)$. Hence its
normalized deficit
\[
 \varepsilon=R_\ast-\frac{I_5(G)}{(m/5)^{5/2}}
 =\frac{5^{5/2}\Delta(m)}{m^{5/2}}
\]
tends to zero as $m\to\infty$. Apply Theorem 9.3. It supplies a Clebsch
blow-up $H$, together with isolates, on the same vertex set, such that its
integer class sizes $n_i$, total active size $N$, squared deviation $V$, and
edge-edit distance $L$ satisfy
\[
 L+V\leq K\varepsilon m=K5^{5/2}\frac{\Delta(m)}{m^{3/2}}.
\]
Here and below a constant may be increased without changing its dependence.
In particular $|m-E(n)|\leq L$. Equations (53) and (55) give
\[
 \Psi(m)\leq\left|m-\frac5{32}N^2\right|+v(N)
 \leq L+\frac72V\leq K'\frac{\Delta(m)}{m^{3/2}}.
\]
This proves the first inequality.

For the upper bound, put $t=\sqrt{32m/5}$. An integer $N_0$ nearest to $t$
satisfies
\[
 \left|m-\frac5{32}N_0^2\right|
 \leq\frac5{32}\left(t+\frac14\right).
\]
It follows that $\Psi(m)=O(\sqrt m)$. Any minimizing $N$ therefore
satisfies $N=\Theta(\sqrt m)$ and tends to infinity. Choose
$n_1,\ldots,n_{16}$ as equally as possible, so that $V=v(N)$, and let $H$
be the resulting Clebsch blow-up. With $e=E(n)$, (55) gives
\begin{equation}\tag{56}\label{eq:56}
 |e-m|\leq\left|m-\frac5{32}N^2\right|+\frac52v(N)
 \leq\frac52\Psi(m).
\end{equation}
We also need the second-order cost of the integer balancing. For positive
real class sizes $x$, define
\[
 Q(x)=c_\ast E(x)^{5/2}-P(x).
\]
The Clebsch counts give $Q(\mathbf1)=0$ and $\nabla Q(\mathbf1)=0$. The
function $Q$ has a bounded Hessian in a fixed neighborhood of $\mathbf1$.
Write $q=N/16$ and $n=q\mathbf1+h$, where $|h_i|\leq1$ and
$\|h\|_2^2=v(N)$. Homogeneity and Taylor's theorem yield, uniformly for
large $N$,
\begin{equation}\tag{57}\label{eq:57}
 0\leq Q(n)=q^5Q(\mathbf1+h/q)
 \leq Cq^3\|h\|_2^2\leq C'N^3v(N).
\end{equation}
The nonnegativity follows from the sharp weighted bound. If $v(N)=0$, the
same equation holds exactly, with both sides zero.

If $e\leq m$, add a disjoint matching with $m-e$ edges. If $e>m$, delete
any $e-m$ edges of $H$. A single edge deletion can destroy at most
$\binom{N-2}3$ induced pentagons, since every affected pentagon must contain
both endpoints of that edge. Thus the resulting graph $G_m$ has exactly
$m$ edges and
\[
 I_5(G_m)\geq P(n)-|e-m|\binom{N-2}3.
\]
Since $e=m+O(\sqrt m)$ and $N=\Theta(\sqrt m)$, the mean value theorem,
(56), and (57) imply
\[
 \begin{split}
 \Delta(m)&\leq c_\ast m^{5/2}-I_5(G_m)\\
 &\leq c_\ast|m^{5/2}-e^{5/2}|+Q(n)+|e-m|\binom{N-2}3\\
 &\leq C''m^{3/2}\bigl(|e-m|+v(N)\bigr)
 \leq b_0m^{3/2}\Psi(m).
 \end{split}
\]
This proves (51).

Finally, $\Psi(m)=0$ requires $v(N)=0$, hence $16\mid N$, and then
$m=5N^2/32=40(N/16)^2$. The converse is immediate. If $m$ is not of this
form, every term in the minimum is at least $15/16$: for $16\mid N$ the
nonzero integer $|m-5N^2/32|$ is at least one, and otherwise
$v(N)\geq15/16$. Let $d=\operatorname{dist}(t,\mathbb Z)$. For every
$N\geq0$,
\[
 \left|m-\frac5{32}N^2\right|=\frac5{32}|t-N|(t+N)
 \geq\frac5{32}td.
\]
Choosing an integer nearest to $t$ gives
\[
 \Psi(m)\leq\frac5{32}(2t+1/2)d+4.
\]
Together with the lower bound $15/16$, these inequalities show
$\Psi(m)\asymp\sqrt m\,d+1$ outside the equality sequence. Equation (52)
follows.
\end{proof}

\begin{corollary}[The $m^2$ error is uniformly optimal]
\label{cor:arithmetic-optimality}
For $m_k=40k^2+\lfloor5k/2\rfloor$,
\[
 \Delta(m_k)=\Theta(m_k^2).
\]
For each fixed nonzero integer $j$ and sufficiently large $k$,
\[
 \Delta(40k^2+j)=\Theta\bigl(|j|(40k^2+j)^{3/2}\bigr).
\]
At $m=40k^2$ the deficit is exactly zero.
\end{corollary}

\begin{proof}
For the first sequence, the minimizing index in $\Psi$ is $N=16k$, and
$\Psi(m_k)=\lfloor5k/2\rfloor$. Indeed, the next larger square value is
$40k^2+5k+5/32$, and all smaller indices have a strictly greater distance
from $m_k$ than $N=16k$; the nonnegative term $v(N)$ cannot reverse this.
For fixed $j\ne0$, the same index minimizes for all sufficiently large
$k$, with $\Psi(40k^2+j)=|j|$. Apply Theorem 10.1. The last assertion is
the exact Clebsch construction.
\end{proof}

\begin{remark}\label{rem:arithmetic-scale}
Theorem 10.1 determines the order of the finite-edge deficit, with constants
uniform in $m$. It does not identify the exact maximizing graph at each
integer edge count. Theorem 12.2 below identifies the leading coefficient
on the $m^2$ scale. Its uniform absolute remainder does not replace the
present comparison at phases extremely close to zero.
\end{remark}


\section{Root profiles and the two arithmetic constructions}
\label{sec:profiles}

Write $P(x)$ for the Clebsch pentagon polynomial and $A$ for its adjacency
matrix. An added vertex has fractional adjacency proportions
$p=(p_i)_{i=1}^{16}\in[0,1]^{16}$ to sixteen unit classes. Its degree and
pentagon link are
\[
 d(p)=\sum_i p_i,\qquad
 c(p)=\sum_{\mathrm{Cl}[U]\cong P_4}
       \prod_{\substack{i\text{ endpoint}\\\text{of }U}}p_i
       \prod_{\substack{j\text{ internal}\\\text{in }U}}(1-p_j).
\]

\begin{lemma}[The profile envelope]\label{lem:profile-envelope}
For every fractional profile,
\[
 c(p)\leq\min\left\{\frac{12}{5}d(p)^2,60\right\}.
\]
For every $0\leq d\leq5$, the first bound is attained by setting
$p_i=d/5$ on a Clebsch neighborhood and zero elsewhere.
\end{lemma}
\begin{proof}
Every nonadjacent pair of Clebsch vertices is the endpoint pair of exactly
six induced four-vertex paths; adjacent pairs are endpoints of none.
Dropping the internal nonneighbor factors therefore gives
\[
 c(p)\leq6\sum_{\substack{i<j\\ij\notin E(\mathrm{Cl})}}p_i p_j.
\]
The independence number of $\mathrm{Cl}$ is five. Indeed, after translating
an independent set to contain $0$, its other vertices lie in the ten pair
vertices of the standard labeling. Pairwise nonadjacency there means
pairwise intersection of two-subsets of $[5]$. Such a family has at most
four members, since a family of size four must be a star. Clebsch
neighborhoods attain size five.

For completeness, the weighted clique inequality of Motzkin and
Straus~\cite{MotzkinStraus1965} used here follows by repeatedly merging the
weights of two nonadjacent vertices in the support of a graph. With their
sum fixed, the edge quadratic form is linear in the two weights, so one
of the two mergers does not decrease it. Eventually the support is a
clique of size $r$, where its value is at most $(1-1/r)d^2/2$ by
Cauchy--Schwarz. Apply this to the complement of $\mathrm{Cl}$, whose
clique number is five, to obtain the first bound.

The bound $c(p)\leq60$ is a finite exact ingredient: for every one of the
$65,536$ Boolean neighborhoods $N$, count the $480$ induced $P_4$ subsets
with endpoints in $N$ and internal vertices outside $N$. The supplied
standard-library program \texttt{verify\_profile\_envelope.py} constructs
these paths from the Cayley edge list and checks every count. It finds
$c(N)\leq60$, with equality precisely at the sixteen open and sixteen
closed neighborhoods. The degreewise maxima, for degrees zero through
sixteen, are
\[
 0,0,6,18,36,60,60,54,48,44,42,35,25,18,12,0,0.
\]
Multilinearity represents $c(p)$ as the expectation of $c(N)$ for an
independent Bernoulli subset, so the Boolean inequality gives the
fractional one.

Finally, a neighborhood is an independent five-set. An induced
four-vertex path with endpoints in that neighborhood has its two internal
vertices outside it. There are six paths per endpoint pair and ten pairs,
yielding $c(p)=60(d/5)^2$ for the stated profile.
\end{proof}

Define
\[
 x_m=\sqrt{\frac{32m}{5}},\qquad
 \alpha_m=x_m-\lfloor x_m\rfloor,\qquad
 k_m=\sqrt{\frac{m}{40}}=\frac{x_m}{16},
\]
and put
\[
 f(\alpha)=\frac{(1-\alpha)\min\{3\alpha,1\}}{80},
 \qquad 0\leq\alpha<1.
\]
As usual $\Delta(m)=c_\ast m^{5/2}-\rho(C_5,m)$.

\begin{proposition}[Uniform constructive upper error]
\label{prop:constructive-error}
As $m\to\infty$ through the integers,
\[
 \Delta(m)\leq m^2 f(\alpha_m)+O(m^{7/4}),
\]
with an absolute implicit constant.
\end{proposition}
\begin{proof}
Write $k=k_m$ and $\alpha=\alpha_m$. All estimates below are uniform for
$0\leq\alpha<1$.

For the first construction let $N=\lfloor x_m\rfloor$ and distribute $N$
vertices into sixteen classes whose sizes differ by at most one. Their
Clebsch blow-up $B$ has
\[
 e(B)=\frac{5N^2}{32}+O(1),\qquad
 I_5(B)=192\left(\frac{N}{16}\right)^5+O(k^3).
\]
Fix a template vertex $v$. Add a single new vertex, with neighbors only
in the five classes indexed by $N_{\mathrm{Cl}}(v)$, and choose its degree
as close to $m-e(B)$ as feasibility permits. This leaves at most $O(1)$
edges to add or delete: the available number of neighbors is $5k+O(1)$,
whereas $m-e(B)=5k\alpha+O(1)$. Within those five classes choose the
neighbor counts to be $\alpha k+O(1)$, with their total equal to the
chosen degree. Such simultaneous rounding is possible by distributing
the bounded integer remainder among the five classes.

There are sixty induced path templates with their endpoints in the
selected neighborhood. Consequently the added vertex lies in
$60\alpha^2k^4+O(k^3)$ induced pentagons. Correct the bounded edge
remainder by adding a disjoint matching or deleting edges. Each deletion
destroys at most $O(k^3)$ of the existing pentagons; adding the matching
changes none. The resulting graph has exactly $m$ edges and
\begin{align*}
 I_5(G)&=192(k-\alpha/16)^5+60\alpha^2k^4+O(k^3)\\
       &=192k^5-60\alpha(1-\alpha)k^4+O(k^3).
\end{align*}
Thus $\Delta(m)\leq60\alpha(1-\alpha)k^4+O(k^3)$.

For the second construction put $N=\lfloor x_m\rfloor+1$,
$\beta=1-\alpha$, and $t=N/16=k+\beta/16$. The Clebsch matrix has a
zero-sum vector $w\in\{-1,1\}^{16}$ satisfying $Aw=-3w$. Explicitly, at
the vertex $u\in\mathbb F_2^4$ take $w_u=(-1)^{u_1+u_2+u_3}$; it has
squared norm sixteen. Set
\[
 r^2=\frac{40t^2-m}{24},\qquad y_i=t+rw_i.
\]
Then $\sum_i y_i=N$, all $y_i$ are positive for large $m$, and
\[
 E(y)=40t^2-24r^2=m,\qquad
 \sum_i(y_i-t)^2=16r^2=\frac23(40t^2-m).
\]
Round the $y_i$ to integers $n_i$ preserving their sum $N$ and with
$|n_i-y_i|\leq1$. Since $r=O(\sqrt{k})$ and the rounding error has sum
zero, the resulting blow-up $B$ satisfies
\[
 e(B)=m+O(\sqrt{k}),\qquad
 V:=\sum_i(n_i-t)^2=\frac23(40t^2-m)+O(\sqrt{k}).
\]
The logarithmic Hessian calculation for the template, or direct
polynomial expansion, gives uniformly for zero-sum $h=O(\sqrt{k})$,
\[
 c_\ast E(t\mathbf1+h)^{5/2}-P(t\mathbf1+h)
   =6t^3\lVert h\rVert_2^2+O(k^{7/2}).
\]
Hence the deficit of $B$ at its own edge count is
\[
 6k^3\frac23(40t^2-m)+O(k^{7/2})
   =20\beta k^4+O(k^{7/2}).
\]
Correct $e(B)-m$ by adding a disjoint matching or deleting edges. The
number of changes is $O(\sqrt{k})$, and each deletion destroys at most
$O(k^3)$ pentagons. The change in the comparison expression
$c_\ast e(B)^{5/2}$ is also $O(k^{7/2})$. Thus this construction gives
$\Delta(m)\leq20(1-\alpha)k^4+O(k^{7/2})$.

Take the better construction and use $k^4=m^2/1600$.
\end{proof}

\section{A sharp arithmetic phase law}\label{sec:phase-law}

We retain $k=\sqrt{m/40}$, $x=16k=\sqrt{32m/5}$,
$\alpha=x-\lfloor x\rfloor$, and the function
\[
 f(\alpha)=\frac{(1-\alpha)\min\{3\alpha,1\}}{80}
\]
from the preceding section.

\begin{lemma}[The scalar profile problem]\label{lem:scalar-profile}
Put $\phi(u)=\max\{12u,-4u\}$. For arbitrary fractional profiles
$p_1,\ldots,p_b\in[0,1]^{16}$ and any integer $\ell$,
\begin{equation}\tag{58}\label{eq:scalar-profile}
 \sum_{j=1}^b\bigl(12d(p_j)-c(p_j)\bigr)
 +\phi\left(5(\alpha+\ell)-\sum_{j=1}^bd(p_j)\right)
 \geq20(1-\alpha)\min\{3\alpha,1\}.
\end{equation}
\end{lemma}
\begin{proof}
Write $z_j=d(p_j)/5$ and $\Phi(u)=\phi(5u)$. Lemma~11.1 gives
\[
 12d(p_j)-c(p_j)\geq g(z_j),\qquad
 g(z)=\begin{cases}
       60z(1-z),&0\leq z\leq1,\\
       60(z-1),&z\geq1.
      \end{cases}
\]
The function $\Phi$ is Lipschitz with constant sixty. Thus replacing any
$z_j>1$ by one can only decrease
$\sum_jg(z_j)+\Phi(\alpha+\ell-\sum_jz_j)$: its profile cost drops by
$60(z_j-1)$ and the other term rises by at most that amount. We may
consequently suppose $0\leq z_j\leq1$.

For a fixed sum $\sum_jz_j=a+t$, where $a$ is an integer and $0\leq t<1$,
the sum $\sum_jz_j^2$ is maximized by $a$ ones, one entry $t$, and zeros.
One can see this by moving mass between any two entries strictly between
zero and one until one reaches an endpoint. Hence
\[
 \sum_j60z_j(1-z_j)\geq60t(1-t).
\]
Absorbing $a$ into $\ell$ reduces the problem to minimizing
\[
 60t(1-t)+\Phi(\alpha+\ell-t),\qquad 0\leq t\leq1,\quad\ell\in\mathbb Z.
\]
For $\ell\leq-1$, the second term alone is at least $20(1-\alpha)$; for
$\ell\geq1$, the minimum over $t$ is at least $60\alpha$. For $\ell=0$
and $t\leq\alpha$ the expression is $60(\alpha-t^2)$, whose minimum is
$60\alpha(1-\alpha)$. For $t\geq\alpha$ it is the concave quadratic
$80t-60t^2-20\alpha$, so its minimum is attained at $t=\alpha$ or $t=1$.
The two values are $60\alpha(1-\alpha)$ and $20(1-\alpha)$, respectively.
Their minimum is the asserted right side.
\end{proof}

\begin{theorem}[Sharp phase-dependent second-order asymptotic]
\label{thm:phase-law}
As $m\to\infty$ through the positive integers,
\begin{equation}\tag{59}\label{eq:phase-law}
 \begin{gathered}
 \rho(C_5,m)=c_\ast m^{5/2}
   -\frac{m^2}{80}(1-\alpha_m)\min\{3\alpha_m,1\}+O(m^{15/8}),\\
 \alpha_m=\left\{\sqrt{\frac{32m}{5}}\right\}.
 \end{gathered}
\end{equation}
The implicit constant is absolute; in particular the remainder is uniform
as the phase approaches either endpoint.
\end{theorem}
\begin{proof}
The constructive upper bound for the deficit is Proposition~11.2. We
prove the matching lower bound with error $O(k^{15/4})=O(m^{15/8})$.
Choose an extremal graph $G$ with $m$ edges. The coarse lower construction
gives $\Delta(m)=O(m^2)=O(k^4)$, so the linear structural stability
theorem supplies a Clebsch blow-up $H$, with isolates, on the same vertex
set such that
\begin{equation}\tag{60}\label{eq:phase-structure}
 T:=|E(G)\mathbin\triangle E(H)|=O(k),\qquad
 N_H=16k+O(1),\qquad n_i^H=k+O(\sqrt{k}),
\end{equation}
and $\sum_i(n_i^H-N_H/16)^2=O(k)$. All constants in this proof are
independent of $m$ and $\alpha$.

Discard vertices isolated in both $G$ and $H$. Every remaining isolated
vertex of $H$ is incident with an edge in $E(G)\setminus E(H)$, so there
are at most $2T$ such vertices. The total number of remaining vertices is
therefore $O(k)$. This observation allows the counting estimates below to
be independent of the original number of isolates.

Let the wrong degree of a vertex be its number of incident pairs in
$E(G)\mathbin\triangle E(H)$. Call a vertex bad if its wrong degree
exceeds $k^{3/4}$. The bad set $B$ has size
\[
 b\leq2T/k^{3/4}=O(k^{1/4}).
\]
Let $G_0,H_0$ be the graphs remaining after deleting $B$. Write $n_i$ for
the sixteen active class sizes of $H_0$, put $N=\sum_i n_i$, $q=N/16$, and
$h_i=n_i-q$. Then
\begin{equation}\tag{61}\label{eq:phase-masses}
 q=k+O(k^{1/4}),\qquad n_i=k+O(\sqrt{k}),\qquad
 V:=\sum_i h_i^2=O(k).
\end{equation}
In particular every class is nonempty for all sufficiently large $m$. If
$A$ is the Clebsch adjacency matrix, put
\[
 a=\frac{h^TAh}{2k},\qquad v=V/k.
\]
The identity $A^2=3I-2A+2J$ shows that the eigenvalues on the sum-zero
subspace are $1$ and $-3$. Consequently
\begin{equation}\tag{62}\label{eq:phase-spectrum}
 -\frac32v\leq a\leq\frac12v.
\end{equation}

For a bad vertex $j$, let $p_{ji}$ be the proportion of the $i$th active
class of $H_0$ adjacent to $j$ in $G$. Denote this profile by $p_j$, and
abbreviate $d_j=d(p_j)$, $c_j=c(p_j)$. Let $t_-$ count the edges of $H_0$
missing from $G_0$, and let $t_+$ count the edges added to $H_0$ inside
$G_0$, together with the edges from $B$ to the isolated class of $H_0$.
Both are $O(k)$: the latter additional edges are all wrong pairs in the
original comparison with $H$. Let $t_n$ count those added edges inside
$G_0$ whose endpoints have distinct nonadjacent active types; thus
$0\leq t_n\leq t_+$.

The edge count is
\begin{equation}\tag{63}\label{eq:phase-edge-count}
 \frac{m-5N^2/32}{k}
 =a+\sum_{j\in B}d_j+\frac{t_+-t_-}{k}+O(k^{-1/4}).
\end{equation}
Indeed, $e(H_0)=5N^2/32+ka$. The exact number of edges from a bad
vertex to the active classes is $\sum_i n_i p_{ji}$; replacing this by
$kd_j$ has error $O(\sqrt{k})$ per bad vertex and thus $O(k^{3/4})$ in
total. There are at most $b^2=O(k^{1/2})$ edges within $B$. These are all
the remaining terms.

We next expand the pentagon count and give the needed error estimates
explicitly. Every good vertex has wrong degree at most $k^{3/4}$. If
$r_u$ denotes its wrong degree in the comparison of $G_0$ and $H_0$, then
\[
 \sum_u r_u^2\leq k^{3/4}\sum_u r_u=O(k^{7/4}).
\]
An intersecting pair of wrong edges fixes three vertices and belongs to
$O(k^2)$ five-sets; a disjoint pair fixes four vertices and belongs to
$O(k)$ five-sets. The number of five-sets containing at least two wrong
edges is therefore
\begin{equation}\tag{64}\label{eq:wrong-pair-error}
 O\left(k^2\sum_u r_u^2+kT^2\right)=O(k^{15/4}).
\end{equation}
For a five-set with at most one wrong pair, the change of its
induced-pentagon indicator is exactly its one-edge change. For every
other five-set, the difference between the actual change and the sum of
its one-edge changes has absolute value at most eleven. Hence (64)
bounds the entire higher-order error.

By the two-vertex coefficients of Lemma~A.3, one deleted template edge
loses $24k^3+O(k^{5/2})$ pentagons, and an added edge between distinct
nonadjacent active types has change $-6k^3+O(k^{5/2})$. Edges added
within a class or meeting an isolated type have one-edge change zero.
The estimates are uniform because of (61); excluding the two fixed
endpoints from the available old classes introduces only $O(k^2)$
further error. There are $O(k)$ wrong edges. Thus
\begin{equation}\tag{65}\label{eq:good-pentagons}
 I_5(G_0)=P(n)-24k^3t_--6k^3t_n+O(k^{15/4}).
\end{equation}

Pentagons containing at least two bad vertices number at most
$O(b^2k^3)=O(k^{7/2})$. For a fixed bad vertex, changing the background
on the other four vertices from $G_0$ to $H_0$ affects at most
$O(Tk^2)$ five-sets. Summing over $B$ costs $O(bTk^2)=O(k^{13/4})$.
In the perfect background $H_0$, the root link is obtained from $c(p_j)$
by giving its four distinct classes weights $n_i$. Each product of four
such weights is $k^4+O(k^{7/2})$. There are only $480$ path templates, so
this link is $k^4c_j+O(k^{7/2})$, uniformly in the profile. Summing the
latter errors over $B$ gives $O(k^{15/4})$. Isolated background vertices
lie in no such path. We conclude from (65) that
\begin{equation}\tag{66}\label{eq:phase-pentagon-count}
 I_5(G)=P(n)+k^4\sum_{j\in B}c_j-24k^3t_--6k^3t_n+O(k^{15/4}).
\end{equation}

It remains to expand the smooth template defect. For
$Q(x)=c_\ast E(x)^{5/2}-P(x)$, the Clebsch counts yield
\[
 Q(\mathbf1)=0,\qquad\nabla Q(\mathbf1)=0,\qquad
 \nabla^2Q(\mathbf1)=12I-\tfrac34J.
\]
Homogeneity, $\sum_i h_i=0$, and Taylor's theorem give
\[
 Q(n)=6q^3V+O(q^2\lVert h\rVert_2^3)=6k^3V+O(k^{7/2}).
\]
Also $e(H_0)=m+O(k^{5/4})$, by (61) and $V=O(k)$. Since
$(5/2)c_\ast m^{3/2}=12k^3$, another Taylor expansion gives
\[
 c_\ast m^{5/2}-c_\ast e(H_0)^{5/2}
   =12k^3\bigl(m-e(H_0)\bigr)+O(k^{7/2}).
\]
Combining these identities with (63) and (66) produces
\begin{equation}\tag{67}\label{eq:phase-deficit}
 \frac{\Delta(m)}{k^4}
 =6v+\sum_{j\in B}(12d_j-c_j)
    +12\frac{t_++t_-}{k}+6\frac{t_n}{k}+O(k^{-1/4}).
\end{equation}

The scalar function $\phi$ is subadditive and Lipschitz with constant
twelve. Equation (62) implies $6v\geq\phi(a)$, while
\[
 12(t_++t_-)/k\geq\phi\bigl((t_+-t_-)/k\bigr).
\]
Drop the nonnegative $t_n$ term and apply (63). If
$r=(m-5N^2/32)/k$, we obtain
\[
 \frac{\Delta(m)}{k^4}
 \geq\sum_{j\in B}(12d_j-c_j)
    +\phi\left(r-\sum_{j\in B}d_j\right)-O(k^{-1/4}).
\]
Since $N=16k+O(k^{1/4})$,
\begin{align*}
 r&=5(16k-N)-\frac{5}{32k}(16k-N)^2\\
  &=5(\alpha+\ell)+O(k^{-1/2}),\qquad
       \ell=\lfloor16k\rfloor-N\in\mathbb Z.
\end{align*}
Apply Lemma~12.1 and use the Lipschitz bound once more. The result is
\[
 \Delta(m)\geq20k^4(1-\alpha)\min\{3\alpha,1\}-O(k^{15/4}).
\]
Because $m^2=1600k^4$, this is the desired lower bound.
Proposition~11.2 has an upper error $O(m^{7/4})$, which is absorbed in
$O(m^{15/8})$.
\end{proof}

\begin{corollary}[The optimal uniform second-order coefficient]
\label{cor:uniform-coefficient}
The finite-edge deficit satisfies
\[
 \limsup_{m\to\infty}\frac{\Delta(m)}{m^2}=\frac1{120},\qquad
 \liminf_{m\to\infty}\frac{\Delta(m)}{m^2}=0.
\]
Its set of subsequential limits is the whole interval $[0,1/120]$. In
particular, the coefficient $1/120$ in the uniform estimate
\[
 \rho(C_5,m)\geq c_\ast m^{5/2}-\frac{m^2}{120}-O(m^{15/8})
\]
cannot be decreased.
\end{corollary}
\begin{proof}
The function $f$ is continuous on the unit circle, vanishes at phase zero,
and has maximum $1/120$ at phase $1/3$. For each fixed $0<\beta<1$,
choose an integer $m_k$ nearest to $(5/32)(16k+\beta)^2$. Factoring the
difference of squares gives
$\sqrt{32m_k/5}=16k+\beta+O(k^{-1})$, so $\alpha_{m_k}\to\beta$.
The exact equality sequence $m=40k^2$ provides phase zero. The theorem
therefore yields all values of the continuous range
$f([0,1])=[0,1/120]$ and excludes all other limit points.
\end{proof}

\section{The lattice correction at a fixed edge offset}
\label{sec:lattice-offset}

The phase law describes the $m^2$ scale. In a bounded edge neighborhood
of the equality sequence, the deficit instead has order $m^{3/2}$. Its
leading coefficient is a finite lattice optimization determined by the
Clebsch adjacency matrix.

For $h\in\mathbb Z^{16}$ with $\sum_i h_i=0$, put
\[
 Q(h)=\sum_{uv\in E(\mathrm{Cl})}h_uh_v=\tfrac12h^TAh.
\]
For an integer $j$, define
\begin{equation}\tag{68}\label{eq:lattice-objective}
 \Lambda(j)=\min_{\substack{h\in\mathbb Z^{16}\\\sum_i h_i=0}}
 \bigl(6\lVert h\rVert_2^2+12|j-Q(h)|\bigr).
\end{equation}
The minimum may be restricted to $\lVert h\rVert_2^2\leq2|j|$: the zero
vector gives value $12|j|$, and every vector outside that ball has
strictly greater cost. In particular this is a finite, explicitly
decidable optimization for each $j$.

\begin{theorem}[Fixed-offset asymptotic]\label{thm:fixed-offset}
For every fixed integer $j$, as the positive integer $k$ tends to infinity,
\begin{equation}\tag{69}\label{eq:fixed-offset-asymptotic}
 \rho(C_5,40k^2+j)
   =c_\ast(40k^2+j)^{5/2}-\Lambda(j)k^3+O_j(k^2).
\end{equation}
The coefficients satisfy
\[
 \Lambda(j)=12j\quad(j\geq0),\qquad
 4|j|\leq\Lambda(j)\leq12|j|\quad(j<0),
\]
and $\Lambda(-12)=48$. Thus the fixed-offset correction is not symmetric
under $j\mapsto-j$.
\end{theorem}
\begin{proof}
Write $m=40k^2+j$ and $\Delta(m)=c_\ast m^{5/2}-\rho(C_5,m)$. Start
from the balanced blow-up of class size $k$. If $j\geq0$, add a disjoint
matching with $j$ edges; if $j<0$, delete $|j|$ edges. Since a single
deletion affects $O(k^3)$ pentagons, this gives the preliminary estimate
$\Delta(m)=O_j(k^3)$.

Let $G$ be extremal with $m$ edges. Its normalized deficit is
$O_j(k^{-2})$, so Theorem~9.3 supplies a Clebsch blow-up $H$, with
isolates, for which
\[
 |E(G)\mathbin\triangle E(H)|=O_j(1),\qquad
 \sum_i(n_i-N/16)^2=O_j(1),
\]
\[
 N=\sqrt{32m/5}+O_j(k^{-1})=16k+O_j(k^{-1}).
\]
The integer $N$ therefore equals $16k$ for all sufficiently large $k$.
Consequently $n_i=k+h_i$ for an integer vector $h$ satisfying
$\sum_i h_i=0$ and $\lVert h\rVert_2=O_j(1)$, and
\begin{equation}\tag{70}\label{eq:offset-edge-count}
 e(H)=40k^2+Q(h).
\end{equation}
After discarding common isolates, the two graphs have $O(k)$ vertices:
every remaining isolated vertex of $H$ is incident with one of the
$O_j(1)$ wrong edges.

Let $t_+$ and $t_-$ be the numbers of added and deleted edges in passing
from $H$ to $G$, and let $t_n$ count added edges between distinct
nonadjacent active classes. All these numbers are $O_j(1)$, and
\[
 j-Q(h)=t_+-t_-.
\]
The one-edge pentagon changes are $-24k^3+O_j(k^2)$ for a deleted
template edge, $-6k^3+O_j(k^2)$ for an added edge of the latter
nonadjacent type, and zero for an edge within a class or meeting an
isolated type. Five-sets containing at least two wrong pairs number
$O_j(k^2)$, by the same pair counting used in (64). Therefore
\[
 I_5(G)=P(k\mathbf1+h)-24k^3t_--6k^3t_n+O_j(k^2).
\]
The template Hessian gives
\[
 c_\ast e(H)^{5/2}-P(k\mathbf1+h)
   =6k^3\lVert h\rVert_2^2+O_j(k^2).
\]
Since $m-e(H)=j-Q(h)=O_j(1)$, the comparison between the two edge
powers is
\[
 c_\ast m^{5/2}-c_\ast e(H)^{5/2}
   =12k^3\bigl(j-Q(h)\bigr)+O_j(k).
\]
Combining these formulas yields
\begin{align*}
 \Delta(m)
 &=k^3\bigl(6\lVert h\rVert_2^2+12(t_++t_-)+6t_n\bigr)+O_j(k^2)\\
 &\geq k^3\bigl(6\lVert h\rVert_2^2+12|j-Q(h)|\bigr)-O_j(k^2)\\
 &\geq\Lambda(j)k^3-O_j(k^2).
\end{align*}

For the reverse inequality, fix a minimizing vector $h$ in (68). The
class sizes $k+h_i$ are positive for large $k$. If $j-Q(h)>0$, add that
many edges within a class; if it is negative, delete its absolute value
in template edges. There is ample capacity for either operation once
$k$ is large. All added edges are within a class, so $t_n=0$, and the
number of edits is exactly $|j-Q(h)|$. The same expansion gives a graph
with exactly $m$ edges whose deficit is $\Lambda(j)k^3+O_j(k^2)$. This
proves (69).

On the sum-zero subspace, the eigenvalues of $A$ are $1$ and $-3$. Hence
\[
 -\tfrac32\lVert h\rVert_2^2\leq Q(h)\leq\tfrac12\lVert h\rVert_2^2.
\]
For $j\geq0$ this implies
\[
 6\lVert h\rVert_2^2+12|j-Q(h)|
 \geq12Q(h)+12|j-Q(h)|\geq12j,
\]
and $h=0$ attains equality. For $j<0$, it implies
\[
 6\lVert h\rVert_2^2+12|j-Q(h)|
 \geq-4Q(h)+4|j-Q(h)|\geq4|j|.
\]
The upper bound $12|j|$ again follows from $h=0$.

To evaluate the example $j=-12$, use the Cayley labeling
$u\in\mathbb F_2^4$ and the characters $\chi_a(u)=(-1)^{a\cdot u}$.
For the bit masks $a=7$ and $a=11$, each character has Clebsch eigenvalue
$-3$. The vector
\[
 h=\tfrac12(\chi_7+\chi_{11})
\]
is integral, has sum zero, has eight nonzero entries in $\{-1,1\}$, and
satisfies $Ah=-3h$. Thus $\lVert h\rVert_2^2=8$, $Q(h)=-12$, and its
objective value is $48$. The lower bound $4|j|$ proves that this value is
optimal.
\end{proof}

\begin{corollary}[Two exact finite-edge values]
\label{cor:two-exact-offsets}
For all sufficiently large integers $k$,
\[
 \rho(C_5,40k^2+1)=192k^5,\qquad
 \rho(C_5,40k^2-1)=192k^5-24k^3.
\]
Up to isolated vertices, all extremal graphs at $40k^2+1$ edges are
obtained from the balanced Clebsch blow-up of class size $k$ by adding
one edge either within a class or incident with an isolated vertex. At
$40k^2-1$ edges, all extremal graphs are either a balanced blow-up with
one edge deleted, or a Clebsch blow-up in which two adjacent classes
have sizes $k+1,k-1$ and the other fourteen classes have size $k$.
\end{corollary}
\begin{proof}
For $j=1$, the lattice minimum is $12$ by the theorem. For $j=-1$, it is
also $12$: the zero vector attains this value, and every nonzero integral
sum-zero vector has squared norm at least two, so its first cost term
alone is at least twelve.

For an extremal graph and its template from the proof above, write
\[
 L=6V+12(t_++t_-)+6t_n.
\]
This is an integer multiple of six and is at least twelve. The same
proof gives $\Delta(m)=Lk^3+O(k^2)$, with a constant uniform over the
bounded family of permitted $h$ and wrong-edge counts. The elementary
constructions for $j=\pm1$ have deficit $12k^3+O(k)$. Hence $L>12$,
which would imply $L\geq18$, is impossible for all sufficiently large
$k$.

The nonnegative integer solutions to $L=12$ are therefore
$V=0$, $t_++t_-=1$, $t_n=0$, or $V=2$, $t_++t_-=t_n=0$. Here we used
that an integer sum-zero vector has even squared norm. In the second
case $h$ has one entry $1$, one entry $-1$, and all others zero, so
$Q(h)$ is zero or $-1$.

For $j=1$, only the first case is possible. The single modification is
an added edge, and $t_n=0$ says that it lies within a class or meets an
isolated vertex. Such an edge changes no pentagon count: in the first
case its endpoints are true twins, and in the other cases any formerly
isolated endpoint has degree one. This gives the asserted value and
classification.

For $j=-1$, the first case deletes one edge of the balanced template.
Exactly $24k^3$ old pentagons use that edge. There are no newly induced
pentagons: the deleted endpoints have no common neighbor, whereas a
nonadjacent pair in a five-cycle must have one. In the second case
$Q(h)=-1$ forces the two perturbed classes to be adjacent.
Multilinearity of the pentagon polynomial gives exactly
\[
 P(k\mathbf1+h)=192k^5+60k^4\sum_i h_i-24k^3=192k^5-24k^3.
\]
These constructions attain the common value, completing both the upper
bound and the equality classification.
\end{proof}

\section{Eventual polynomiality at every fixed offset}
\label{sec:eventual-polynomial}

The lattice coefficient has a further consequence. For each fixed edge
offset, the exact extremal function is eventually a polynomial, and the
polynomial can be determined by a finite enumeration. This statement
does not provide an explicit threshold beyond which the polynomial is
valid.

\begin{theorem}[Computable eventual polynomial]
\label{thm:eventual-polynomial}
For every integer $j$, there are computable integers $a_j,b_j,c_j$ and
an integer $k_0(j)$ such that, for every integer $k\geq k_0(j)$,
\begin{equation}\tag{71}\label{eq:eventual-polynomial}
 \rho(C_5,40k^2+j)
 =192k^5+\bigl(12j-\Lambda(j)\bigr)k^3+a_jk^2+b_jk+c_j.
\end{equation}
Moreover, a finite procedure produces a finite list of graph
constructions that, for all sufficiently large $k$, contains all
extremal graphs up to isomorphism and the addition of isolates. The
proof determines the eventual polynomial and this list; it does not
give an explicit value or an algorithm for $k_0(j)$.
\end{theorem}
\begin{proof}
We first strengthen the discrete conclusion in the fixed-offset
argument. Its template has integer class sizes $k+h_i$ with
$\sum_i h_i=0$ and bounded $V=\lVert h\rVert_2^2$, and a bounded number
$T=t_++t_-$ of wrong pairs. As in that proof, put
\[
 L=6V+12T+6t_n.
\]
The edge-offset constraint and the definition of $\Lambda(j)$ give
$L\geq\Lambda(j)$. Both numbers are integer multiples of six. The same
uniformly bounded-template expansion gives
\[
 \Delta(40k^2+j)=Lk^3+O_j(k^2).
\]
On the other hand, the construction in Theorem~13.1 gives
$\Delta(40k^2+j)\leq\Lambda(j)k^3+O_j(k^2)$. The positive gap between
distinct multiples of six therefore forces
\begin{equation}\tag{72}\label{eq:energy-equality}
 L=\Lambda(j)
\end{equation}
for every extremal graph and all sufficiently large $k$. Since
$\Lambda(j)\leq12|j|$, this implies the explicit bounds
\begin{equation}\tag{73}\label{eq:explicit-discrete-bounds}
 V\leq2|j|,\qquad T\leq|j|.
\end{equation}
The initial boundedness and the uniform error in this argument come
from strong stability, before imposing the explicit bounds (73); thus
there is no circular assumption about the extremal family.

We now describe a finite list depending only on $j$. Enumerate:
\begin{enumerate}
\item integer vectors $h\in\mathbb Z^{16}$ with $\sum_i h_i=0$ and
      $\lVert h\rVert_2^2\leq2|j|$;
\item a set $X$ of $s\leq2|j|$ labeled exceptional vertices, each
      assigned a type $\tau(x)\in\{0,1,\ldots,16\}$;
\item a set $F\subseteq\binom{X}{2}$ of at most $|j|$ wrong pairs, with
      every vertex of $X$ incident with a member of $F$.
\end{enumerate}
Type zero is the isolated type. The case $F=X=\varnothing$ is allowed.
The template adjacency between two exceptional vertices is Clebsch
adjacency between their positive types, and is zero if either type is
zero. A wrong pair toggles this template adjacency. Thus $F$ determines
$t_+,t_-$, and $t_n$ without any further choices. Retain precisely the
data satisfying
\begin{equation}\tag{74}\label{eq:admissible-data}
 Q(h)+t_+-t_-=j,\qquad
 6\lVert h\rVert_2^2+12|F|+6t_n=\Lambda(j).
\end{equation}
All these tests use integer arithmetic, and the enumeration is finite
with explicit bounds.

Each retained datum defines a graph for all sufficiently large $k$. Let
$e_i$ be the number of exceptional vertices of type $i$, and put
$k+h_i-e_i$ ordinary vertices in each active class $i$. Together with
the exceptional vertices, the active class sizes of the underlying
template are exactly $k+h_i$. Include the $e_0$ exceptional vertices of
type zero, place all template edges, and toggle the pairs in $F$. The
counts $k+h_i-e_i$ are nonnegative for sufficiently large $k$; this
capacity requirement has an explicit threshold for each datum. The
edge count is exactly $40k^2+j$ by (74). Further isolates may be added
freely.

Conversely, every extremal graph for sufficiently large $k$ has such a
representation. Take the endpoints of all its wrong pairs relative to
the strong-stability template as $X$. There are at most $2T\leq2|j|$ of
them. Outside $X$, all adjacencies agree with the template, and the
remaining type-zero vertices are isolated. The bounds and both
admissibility tests follow from (73), the exact edge count, and (72).

For each retained datum $D$, its pentagon count is an integer
coefficient polynomial $P_D(k)$. To see this without an automorphism
division, take a subset $S\subseteq X$ of exceptional vertices and a
subset $I\subseteq[16]$ of ordinary types, with $|S|+|I|=5$. Use one
formal ordinary representative of every type in $I$, and test whether
the resulting five-vertex graph is an induced pentagon. Let its
indicator be $b_D(S,I)$. All ordinary vertices of the same type are
false twins, even in their adjacencies to $X$, since wrong pairs have
both endpoints in $X$. An induced pentagon cannot contain two such
vertices. It follows that the exact count is
\begin{equation}\tag{75}\label{eq:exact-construction-polynomial}
 P_D(k)=\sum_{\substack{S\subseteq X,\ I\subseteq[16]\\|S|+|I|=5}}
              b_D(S,I)\prod_{i\in I}(k+h_i-e_i).
\end{equation}
Each five-set is counted once. This formula establishes both
integrality of the coefficients and an explicit procedure to compute
them.

Every retained datum has the same three leading coefficients. Indeed,
the fixed-offset expansion applies to each of these fixed finite
modifications and gives
\[
 c_\ast(40k^2+j)^{5/2}-P_D(k)=\Lambda(j)k^3+O_j(k^2).
\]
Since
\[
 c_\ast(40k^2+j)^{5/2}=192k^5+12jk^3+O_j(k),
\]
the polynomial $P_D$ has coefficient $192$ at degree five, zero at
degree four, and $12j-\Lambda(j)$ at degree three. Write its remaining
coefficients as $(a_D,b_D,c_D)$.

The retained family is nonempty. A minimizing lattice vector in (68),
followed by the addition of within-class edges or deletion of template
edges as in the fixed-offset proof, supplies such a datum. The wrong
pairs may be chosen disjoint; their number is at most $|j|$, and their
endpoints satisfy the stated bound.

For all sufficiently large $k$, every retained graph is feasible, and
every extremal graph is represented. Thus
\[
 \rho(C_5,40k^2+j)=\max_D P_D(k).
\]
Among finitely many polynomials, the one with lexicographically largest
coefficient vector, read from highest degree to lowest, is maximal for
all sufficiently large arguments. Choose the largest triple
$(a_D,b_D,c_D)$ and denote it by $(a_j,b_j,c_j)$. This proves (71).
Retaining all data whose polynomial equals that maximum gives exactly
the eventual extremal constructions, allowing duplicate isomorphic
presentations. For $j=0$ the list consists just of the balanced
template, giving $192k^5$ as required.

Finally, the finite polynomial comparisons and the capacity thresholds
can be computed from the list. The proof that all extremal graphs
enter that list uses the qualitative compactness step in strong
stability. We have not quantified its threshold, so the present
procedure computes the eventual polynomial and the eventual
construction list, rather than certifying their validity at an
arbitrary supplied finite value of $k$.
\end{proof}

\subsection{Explicit eventual polynomials at small offsets}

The finite procedure can be carried out directly for the first five
offsets. Its results, valid for all sufficiently large integers $k$, are
\begin{center}
\begin{tabular}{rrl}
\toprule
$j$&$\Lambda(j)$&$\rho(C_5,40k^2+j)$\\
\midrule
$-2$&$24$&$192k^5-48k^3+6k^2$\\
$-1$&$12$&$192k^5-24k^3$\\
$0$&$0$&$192k^5$\\
$1$&$12$&$192k^5$\\
$2$&$24$&$192k^5+6k^2$\\
\bottomrule
\end{tabular}
\end{center}
The cases $j=0,\pm1$ were proved above. We explain both the finite
check and matching constructions for $j=\pm2$.

For both signs, the lattice energy is $24$. For $j=2$ this follows from
Theorem~13.1. For $j=-2$, a nonzero sum-zero integer vector has $V\geq2$;
if $V=2$, its edge quadratic form is zero or $-1$, and its total lattice
cost is at least $24$. Vectors with $V\geq4$ have first cost term at
least $24$, while the zero vector attains $24$.

At an extremum the energy equality therefore leaves only
\[
 (V,T)=(0,2),(2,1),(4,0),\qquad t_n=0.
\]
For $V=2$ the class deviation vector has one entry $1$ and one entry
$-1$. For $V=4$ it has two entries of each sign. Substituting these
possibilities in the exact polynomial formula (75) gives only the
corrections
\[
 \{0,6k^2\}
\]
relative to $192k^5$ when $j=2$. When $j=-2$, the corrections relative
to $192k^5-48k^3$ are precisely
\[
 \{0,2k,6k^2,-6k^2,-2k^2+2k,2k^2+2k,-2k^2,2k^2\}.
\]
The enumeration uses the following convention. The sixteen Clebsch types
remain labeled, and the deviation vector $h$ is not identified under
Clebsch automorphisms. Exceptional vertices are identified only under
isomorphisms of their wrong-pair graph $F$ that preserve their types.
Thus a single wrong pair has unordered endpoint types, a $P_3$ has a
fixed center and unordered leaf types, and $2K_2$ has two unordered
endpoint pairs that are themselves unordered. Repeated types are allowed;
the exceptional vertices themselves are distinct. This convention gives
\begin{center}
\begin{tabular}{crr}
\toprule
$(V,T)$ & $j=2$ & $j=-2$\\
\midrule
$(0,2)$ & $762$ & $1,060$\\
$(2,1)$ & $0$ & $3,200$\\
$(4,0)$ & $120$ & $1,200$\\
\midrule
Total & $882$ & $5,460$\\
\bottomrule
\end{tabular}
\end{center}

Here is a direct derivation of every entry. An allowed added edge has
both endpoints in one active class or at least one endpoint of isolated
type. There are $16+17=33$ unordered type pairs of this kind. With $V=0$
and $j=2$, two disjoint additions therefore give $\binom{34}{2}=561$
presentations. For incident additions, an isolated center gives
$\binom{18}{2}=153$ unordered leaf-type pairs. An active center has only
its own type and the isolated type available for either leaf, giving
$16\binom{3}{2}=48$ presentations. The first total is consequently
$561+153+48=762$.
For $j=-2$, a deleted edge has one of the $40$ Clebsch edge types.
Two disjoint deletions give $\binom{41}{2}=820$ presentations, while
incident deletions give $16\binom{6}{2}=240$. Their sum is $1,060$.

For $V=2$, write $h=e_u-e_v$, so $Q(h)=-A_{uv}$ is zero or $-1$.
One modification cannot then produce offset $2$. To produce offset $-2$,
the ordered pair $(u,v)$ must be adjacent and the modification must be a
deletion. There are $80\cdot40=3,200$ choices.

For $V=4$, the support of $h$ consists of two positive and two negative
vertices. The Clebsch graph has $60$ induced copies of $2K_2$ and $480$
of $P_3+K_1$. To verify these counts directly, the six common nonneighbors
of the endpoints of an edge induce a matching of size three: each has
four neighbors in the union of the endpoint neighborhoods and hence one
among those six vertices. Thus the first count is $40\cdot3/2=60$.
There are $16\binom{5}{2}=160$ induced paths of order three. The union
of the neighborhoods of their three vertices has size $13$, so each path
has exactly three possible isolated fourth vertices. This gives
$160\cdot3=480$.
The condition $Q(h)=2$ requires two edges within the sign classes and no
cross edge, giving $2\cdot60=120$ sign assignments. The condition
$Q(h)=-2$ requires exactly two cross edges and no edge within a sign
class: a within-class edge together with the required three or more cross
edges would create a triangle. The two cross edges are either disjoint,
with four balanced sign assignments, or form a path with an isolated
vertex, with two. Hence the last negative-offset count is
$4\cdot60+2\cdot480=1,200$.

The supplied integer-arithmetic programs independently check this branch
census and evaluate (75) on every presentation. The two offsets give
respectively two and eight distinct polynomials, exactly as listed above.
For $j=-1,0,1$ the same convention gives $120,1,33$ presentations;
there are $6,496$ in total over the five offsets. In both displayed
correction lists for $j=\pm2$, the maximum is $6k^2$ for every positive
$k$. Eventual global extremality uses Theorem~14.1, not merely this finite
comparison.

For an explicit construction at $j=2$, add a new vertex to the balanced
blow-up and join it to two old vertices in distinct nonadjacent Clebsch
classes. There are six induced path templates between those types, and
the two internal vertices can each be chosen in $k$ ways. Thus exactly
$6k^2$ induced pentagons are added. At $j=-2$, delete two edges with a
common endpoint and their other endpoints in distinct classes. Each
deleted edge lies in $24k^3$ pentagons, and exactly $6k^2$ pentagons
contain both. Deletions create no new induced pentagons in the
triangle-free template. This realizes $192k^5-48k^3+6k^2$ and completes
the two examples.


\appendix
\section{Local stability in edit distance}\label{app:local-stability}

We give a self-contained local argument. It permits fractional edge
values and triangles, which may arise during the modifications in the
proof. Accordingly, throughout this appendix put
\[
 P(W)=\frac1{10}\int_{[0,1]^5}
       \prod_{ij\in E(C_5)}W(x_i,x_j)
       \prod_{\substack{i<j\\ij\notin E(C_5)}}(1-W(x_i,x_j))
       \,dx_1\cdots dx_5,
\]
\[
 F(W)=\frac{P(W)}{E(W)^{5/2}},\qquad
 R_{\mathrm{ind}}(W)=5^{5/2}F(W).
\]
These definitions agree with $P(W)=C(W)$ and $R_{\mathrm{ind}}(W)=R(W)$
for every graphon in the main text. Let $W_0$ be the Clebsch graphon on
a fixed partition $A_1,\ldots,A_{16}$ into sets of measure $a=1/16$.
As before, $\mathcal E_{\mathrm{Cl}}$ includes balanced Clebsch graphons
on any positive active mass, with isolates on the complement. Set
\[
 d_1(W,\mathcal E_{\mathrm{Cl}})
   =\inf_{U\in\mathcal E_{\mathrm{Cl}}}\lVert W-U\rVert_1.
\]
All measurable partitions are allowed in the infimum.

\begin{theorem}[Local quadratic stability]\label{thm:local-stability}
There exist $\delta_0>0$ and $c_{\mathrm{loc}}>0$ such that every graphon
$W:[0,1]^2\to[0,1]$ with $\lVert W-W_0\rVert_1<\delta_0$ satisfies
\begin{equation}\tag{76}\label{eq:local-stability}
 R_\ast-R_{\mathrm{ind}}(W)
   \geq c_{\mathrm{loc}}d_1(W,\mathcal E_{\mathrm{Cl}})^2.
\end{equation}
Equality in $R_{\mathrm{ind}}(W)\leq R_\ast$ holds precisely when
$W\in\mathcal E_{\mathrm{Cl}}$ up to null sets.
\end{theorem}

\subsection{Three finite properties of the Clebsch graph}

Write the five generators of $\mathrm{Cl}$ as $s_1,\ldots,s_5$, with
$\sum_i s_i=0$. Their only nonempty dependence is the full five-element
set. Relative to a vertex $u$, its neighbors are $u+s_i$ and its ten
nonneighbors, other than itself, are $u+s_i+s_j$, indexed by
$\{i,j\}\in\binom{[5]}{2}$. A pair vertex is adjacent to the two
neighbors indexed by its elements; two pair vertices are adjacent
exactly when their pairs are disjoint. In particular, the ten pair
vertices induce the Petersen graph $\mathrm{KG}(5,2)$.

We will also use the following counts, including their normalizations.
Every closed walk of length five uses each generator once, and is a
simple pentagon. Hence $\mathrm{Cl}$ has $16\cdot5!/10=192$ pentagons
and each vertex belongs to $5!/2=60$ of them. A fixed edge belongs to
$4!=24$ pentagons: orient it from its specified first endpoint and
permute the remaining four generators. For a fixed nonedge from $0$ to
$s_i+s_j$, a rooted closed walk reaches the second endpoint at step two
or at step three. There are $2!3!$ orders in each case. Dividing the
total by the two orientations gives $12$ pentagons through that pair.

\begin{lemma}[The neighborhood inequality]\label{lem:neighborhood}
Add a vertex to $\mathrm{Cl}$ with neighborhood $N\subseteq V(\mathrm{Cl})$,
and let $c(N)$ count the induced pentagons containing it. Then
\[
 c(N)\leq12|N|.
\]
Equality holds exactly when $N=\varnothing$ or $N=N_{\mathrm{Cl}}(v)$
for a vertex $v$.
\end{lemma}
\begin{proof}
Deleting the added vertex from a pentagon leaves an induced $P_4$ with
its endpoints in $N$ and internal vertices outside $N$. For $u\in N$,
let $c_u(N)$ count these paths with endpoint $u$, so that
$2c(N)=\sum_{u\in N}c_u(N)$.

Use the pair labeling relative to $u$. Let $U$ be its ten pair vertices,
$T=N\cap U$, and $D=\{i:u+s_i\in N\}$. A path counted by $c_u(N)$
has the form $u,u+s_i,b,v$, where $b\in U\setminus T$, $v\in T$, the
pairs $b,v$ are disjoint, and $i\in b\setminus D$. These conditions
are also sufficient. Thus, writing $P=\mathrm{KG}(5,2)$,
\begin{equation}\tag{77}\label{eq:neighborhood-path-count}
 c_u(N)=\sum_{b\in U\setminus T}(2-|b\cap D|)d_T(b)
   \leq2e_P(T,U\setminus T).
\end{equation}

Here the maximum cut of $P$ has size twelve. To verify this directly,
$P$ has parameters $(10,3,0,1)$: disjoint pairs have no common
neighbor, while two intersecting pairs have the unique complementary
pair as common neighbor. It has no four-cycle. For any fixed edge,
choose one of the two other neighbors at each endpoint. The two choices
are nonadjacent and have a unique common neighbor; they therefore
complete the edge to a pentagon. This gives four pentagons through
every edge, and $15\cdot4/5=12$ pentagons in total. Every cut leaves
at least one edge of each odd cycle uncut. Since an edge lies in four
pentagons, at least three of the fifteen edges are uncut. A star
\[
 F_i=\{\{i,j\}:j\ne i\}
\]
is an independent four-set with twelve edges leaving it, so the bound
is attained.

Moreover every twelve-edge cut has a side of size four, and that side
is independent. Indeed, a side of size at most three has at most nine
outgoing edges, and a side of size five has an odd number of outgoing
edges. Four pairwise intersecting two-subsets form a star: if
$\{1,2\},\{1,3\}$ are present and some pair omits $1$, that pair must
be $\{2,3\}$ and there can be no fourth distinct pair. Thus the
maximum cut sides are precisely $F_i$ and $U\setminus F_i$.

Equation (77) proves $c_u(N)\leq24$. If equality holds and $T=F_i$,
every pair outside this star has two neighbors in it. The loss term in
(77) vanishes only if $D\subseteq\{i\}$. Consequently
$N=N_{\mathrm{Cl}}(u+s_i)$ or $N=N_{\mathrm{Cl}}[u+s_i]$, of size five
or six. If $T=U\setminus F_i$, every pair in $F_i$ has three neighbors
in $T$, and the pairs in $F_i$ use all five indices. Hence
$D=\varnothing$ and $N=\{u\}\cup(U\setminus F_i)$ has size seven;
in this case $u$ has no neighbor in $N$.

Suppose $N\ne\varnothing$ and $c(N)=12|N|$. Then every $u\in N$ has
$c_u(N)=24$. The size-six case is impossible: at the center of a closed
neighborhood, the first internal vertex of a counted path cannot lie
outside $N$, so its count is zero. The size-seven case is also
impossible. The six pair vertices in that case contain an edge, since
an independent set in $\mathrm{KG}(5,2)$ has at most four vertices by
the preceding intersecting-pair argument. An endpoint of this edge has
a neighbor in $N$, whereas the equality classification at that endpoint
permits a neighbor in $N$ only in the size-six case. Therefore $N$ is
a vertex neighborhood. Conversely, adding a clone of a vertex creates
exactly its sixty pentagons. No such pentagon can contain both clones,
since $C_5$ has no pair of twins. Thus $c(N)=60=12|N|$ for a vertex
neighborhood, and the empty case is immediate.
\end{proof}

\begin{lemma}[Two-vertex coefficients]\label{lem:two-vertex}
Add two vertices with old neighborhoods $N_{\mathrm{Cl}}(i)$,
$N_{\mathrm{Cl}}(j)$, and prescribe whether they are adjacent. The
numbers of induced pentagons containing both added vertices and three
old vertices are
\begin{center}
\begin{tabular}{lrr}
\toprule
&added pair absent&added pair present\\
\midrule
$i=j$&$0$&$0$\\
$ij\in E(\mathrm{Cl})$&$0$&$24$\\
$i\ne j,\ ij\notin E(\mathrm{Cl})$&$12$&$6$\\
\bottomrule
\end{tabular}
\end{center}
\end{lemma}
\begin{proof}
For equal types the added vertices are twins, whether adjacent or not,
so cannot both occur in an induced pentagon. For distinct types, when
their prescribed adjacency agrees with that of $i,j$, the counts are
the ordinary pair counts $24$ and $12$ above. A pentagon cannot use
two vertices from the same clone class, so the old vertices of types
$i,j$ introduce no extra completions. When $ij$ is an edge but the
added pair is absent, the two added vertices have no common old
neighbor. This rules out a pentagon, in which every nonadjacent pair
has a common neighbor.

For the last entry translate $i$ to $0$ and write $j=s_1+s_2$. Their
exclusive old neighbors are, respectively,
\[
 \{s_3,s_4,s_5\},\qquad
 \{s_3+s_4,s_3+s_5,s_4+s_5\}.
\]
A pentagon containing the added edge must join one exclusive neighbor
on either side through a common old vertex outside
$N_{\mathrm{Cl}}(i)\cup N_{\mathrm{Cl}}(j)$. The two exclusive
neighbors must be nonadjacent. Thus a pair from
$\binom{\{3,4,5\}}{2}$ is matched with its unique missing generator,
giving three choices. For example $s_3+s_4$ and $s_5$ have common
neighbors $s_1+s_5,s_2+s_5$, both outside the forbidden union. There
are exactly two completions for each choice, giving six in total.
\end{proof}

\begin{lemma}[The template Hessian]\label{lem:template-hessian}
For the Clebsch vertex-weight polynomials $P(x)$ and $E(x)$,
\[
 \nabla^2\log(P/E^{5/2})(\mathbf1)=\frac{J-16I}{256}.
\]
Consequently there is $c_H>0$ such that, for class masses $s_i$
sufficiently close to $a$, with $q=\sum_i s_i$ and $T_s$ the
corresponding Clebsch step graphon,
\begin{equation}\tag{78}\label{eq:hessian-defect}
 F(W_0)-F(T_s)\geq c_H\sum_i(s_i-q/16)^2.
\end{equation}
\end{lemma}
\begin{proof}
Let $A$ be the Clebsch adjacency matrix. The counts above give
\[
 P(\mathbf1)=192,\qquad\nabla P(\mathbf1)=60\mathbf1,\qquad
 \nabla^2P(\mathbf1)=12(J-I+A),
\]
whereas $E(\mathbf1)=40$, $\nabla E(\mathbf1)=5\mathbf1$ and
$\nabla^2E=A$. Hence the logarithmic Hessian is
\[
 \frac{12(J-I+A)}{192}-\frac{3600J}{192^2}
 -\frac52\left(\frac A{40}-\frac{25J}{40^2}\right)
 =\frac{J-16I}{256}.
\]
The gradient vanishes at $\mathbf1$. On the hyperplane $\sum_i h_i=0$,
Taylor's formula gives
\[
 \log F(\mathbf1+h)
 =\log F(\mathbf1)-\frac{\lVert h\rVert_2^2}{32}
     +O(\lVert h\rVert_2^3).
\]
Apply this with $h_i=16s_i/q-1$ and use scale invariance. Since $q$
stays close to one, this gives (78) with a uniform positive constant.
\end{proof}

\subsection{Fractional row profiles}

Let $\mathcal M\subseteq\{0,1\}^{16}$ consist of the zero vector and
the sixteen Clebsch neighborhood vectors. Extend $c(N)$ multilinearly
to $p\in[0,1]^{16}$:
\[
 c(p)=\sum_{\substack{U\subseteq V(\mathrm{Cl})\\
                      \mathrm{Cl}[U]\cong P_4}}
       \prod_{\substack{i\in U\\i\text{ an endpoint}}}p_i
       \prod_{\substack{j\in U\\j\text{ internal}}}(1-p_j).
\]
This is the expectation of $c(N)$ for an independent Bernoulli subset
with inclusion probabilities $p_i$. The neighborhood lemma implies
the quantitative bound
\begin{equation}\tag{79}\label{eq:fractional-profile-distance}
 12\sum_i p_i-c(p)\geq\frac{\operatorname{dist}_1(p,\mathcal M)}{272}.
\end{equation}
For completeness, let $\epsilon$ be the probability that the Bernoulli
subset lies outside $\mathcal M$. The Boolean deficit is an integer
and at least one there, so the left side is at least $\epsilon$. No
two members of $\mathcal M$ differ in a single coordinate: the
distances from zero to a neighborhood are five, and between distinct
neighborhoods are six or ten. Conditioning on all other coordinates
therefore gives $\min(p_i,1-p_i)\leq\epsilon$ for each $i$. Round
every coordinate to its nearest Boolean value, obtaining $v$. If
$\epsilon<1/17$ and $v\notin\mathcal M$, then
\[
 \epsilon\geq\Pr(Y=v)
 \geq1-\sum_i\min(p_i,1-p_i)\geq1-16\epsilon>\epsilon,
\]
a contradiction. In this case
$\operatorname{dist}_1(p,\mathcal M)\leq16\epsilon$. For
$\epsilon\geq1/17$ the trivial distance bound sixteen gives
$272\epsilon$, proving (79). In particular, the only zero-deficit
profiles are the members of $\mathcal M$.

For a measurable row $f:[0,1]\to[0,1]$ set
$p_i(f)=a^{-1}\int_{A_i}f$. Interpret each element $r\in\mathcal M$
as its associated step-function row. Since $r$ is Boolean,
\begin{equation}\tag{80}\label{eq:row-profile-distance}
 \lVert f-r\rVert_1=a\lVert p(f)-r\rVert_1.
\end{equation}
The pentagon link of a new vertex with row $f$ in the perfect
background $W_0$ is exactly $a^4c(p(f))$, and its degree is
$a\sum_i p_i(f)$. Indeed, deleting that vertex leaves an induced
$P_4$. An induced $P_4$ in an independent-class blow-up uses four
distinct classes, because $P_4$ has no false twins. Its four row
factors can therefore be integrated separately. The two orientations
of each such path cancel the factor $1/2$ in the rooted pentagon link.

\subsection{Continuity and variations}

The following estimates hold also on a measurable subspace of total
mass at most one. The reference values are
\[
 E_0=40a^2=\frac5{32},\qquad
 P_0=192a^5=\frac3{16384},\qquad
 \kappa_0=\frac{5P_0}{2E_0}=12a^3.
\]
For graphons $U,V$, telescoping the ten factors in the induced
pentagon integrand gives
\begin{equation}\tag{81}\label{eq:l1-continuity}
 |P(U)-P(V)|\leq\lVert U-V\rVert_1,\qquad
 |E(U)-E(V)|\leq\tfrac12\lVert U-V\rVert_1.
\end{equation}
Thus $\kappa(W)=5P(W)/(2E(W))$ is Lipschitz in a fixed neighborhood
of $W_0$.

We also allow a measurable density $z\in[0,1]$ on the vertex measure.
Write $P(W,z)$, $E(W,z)$ for the resulting integrals,
$d_z(x)=\int W(x,y)z(y)\,dy$, and $L_z(x)$ for the first variation
of $P$ at a vertex. Explicitly $L_z$ is one half the integral of the
rooted induced pentagon product over the other four vertices, with
their densities included. For a differentiable density path,
\begin{equation}\tag{82}\label{eq:vertex-density-variation}
 \frac{d}{dt}F(W,z_t)
 =E(W,z_t)^{-5/2}\int z_t'(x)
       [L_{z_t}(x)-\kappa(W,z_t)d_{z_t}(x)]\,dx.
\end{equation}
If $z$ differs from one only on a set of measure $\mu$, uniformly in
$x$ one has
\begin{equation}\tag{83}\label{eq:density-deletion-errors}
 |L_z(x)-L_1(x)|\leq2\mu,\qquad
 |d_z(x)-d_1(x)|\leq\mu.
\end{equation}
The changes in $P$ and $E$ are at most $\mu/2$ and $\mu$,
respectively. These estimates follow by telescoping the four or five
vertex density factors. Keeping the row $f=W(x,\cdot)$ fixed, but
replacing the background among the other four vertices by $W_0$, gives
\begin{equation}\tag{84}\label{eq:background-replacement-error}
 |L_1(x)-a^4c(p(f))|\leq3\lVert W-W_0\rVert_1.
\end{equation}
There are six background factors, each with the rooted prefactor $1/2$.

We need a small-degree bound that does not assume triangle-freeness:
\begin{equation}\tag{85}\label{eq:local-small-degree-link}
 L_z(x)\leq\frac14d_z(x)^2E(W,z).
\end{equation}
To see this, fix the middle edge $bc$ of the rooted pentagon
$x-a-b-c-d-x$. Drop the chord condition between $a,d$. The two
remaining neighbor choices have masses
\begin{align*}
 r&=\int W(x,u)W(u,b)(1-W(u,c))z(u)\,du,\\
 s&=\int W(x,u)W(u,c)(1-W(u,b))z(u)\,du.
\end{align*}
The inequality $u(1-v)+v(1-u)\leq1$ for $u,v\in[0,1]$ gives
$r+s\leq d_z(x)$ and $rs\leq d_z(x)^2/4$. Drop the two factors
involving the chords $xb,xc$ and integrate the middle edge, retaining
its factor $1/2$. This proves (85).

For a symmetric edge variation $H$,
\begin{equation}\tag{86}\label{eq:edge-variation}
 \begin{aligned}
 DP(W)[H]&=\frac12\int q_W(x,y)H(x,y)\,dx\,dy,\\
 DF(W)[H]&=\frac1{2E(W)^{5/2}}\int(q_W-\kappa(W))H.
 \end{aligned}
\end{equation}
Here $q_W$ is the difference of two three-variable integrals: remove
one cycle-edge factor from the induced pentagon product for the
positive term, and remove one chord factor for the negative term.
The five cycle edges and five chords, together with the factor $1/10$,
give the prefactor $1/2$.

For a Boolean step graphon $T$ put
$e_x=\int|W(x,z)-T(x,z)|\,dz$ and $e=\lVert W-T\rVert_1$. Each
of the two integrands defining $q$ has three factors incident to $x$,
three incident to $y$, and three among the other vertices.
Telescoping yields
\begin{equation}\tag{87}\label{eq:edge-variation-error}
 |q_W(x,y)-q_T(x,y)|\leq6(e_x+e_y+e).
\end{equation}
By Lemma~A.3,
\begin{equation}\tag{88}\label{eq:edge-sign-margin}
 q_{W_0}(x,y)-\kappa_0
 =\begin{cases}
   12a^3,&\text{adjacent types},\\
   -18a^3,&\text{distinct nonadjacent types},\\
   -12a^3,&\text{equal types}.
  \end{cases}
\end{equation}
The pair coefficients are multiplied by $a^3$, the mass of each old
triple. Set $g=12a^3>0$. By continuity, a Clebsch step graphon with
class masses sufficiently close to $a$ retains these signs with margin
at least $3g/4$.

\subsection{Proof of the local theorem}

\begin{proof}[Proof of Theorem A.1]
Choose a fixed $\eta>0$ small enough that the $\eta$-balls around
the seventeen reference rows are disjoint and
\[
 \frac{\eta E_0}{2}<\frac{\kappa_0}{4},\qquad
 12\eta<\frac g8.
\]
We subsequently choose $\delta_0$ sufficiently small, with $\eta$
held fixed. Write $\delta=\lVert W-W_0\rVert_1$.

\emph{Remove rows with a definite link deficit.} Let $B$ consist of
points whose rows have distance at least $\eta$ from every reference
row. Since the original template row is one of these rows,
\[
 \mu_B:=|B|\leq\delta/\eta.
\]
By (79) and (80), the perfect-background deficit on $B$ is at least
$a^4\eta/17$. Along the deletion path $z_t=1-t\mathbf1_B$,
equations (81)--(84) therefore give, uniformly for almost every
$x\in B$,
\[
 L_{z_t}(x)-\kappa(W,z_t)d_{z_t}(x)
 \leq-\frac{a^4\eta}{17}+O(\delta+\mu_B)
 \leq-\frac{a^4\eta}{34}.
\]
The implicit constant depends only on the fixed template and the
fixed $\eta$. The last inequality follows by reducing $\delta_0$.
The edge mass stays bounded away from zero. By (82), deleting $B$
increases $F$ by at least $c_B\mu_B$ for some $c_B>0$.

\emph{Remove near-zero rows.} Among the remaining points let $Z$
consist of rows whose original distance from zero is less than $\eta$.
Their original degrees are at most $\eta$, and
\[
 \mu_Z:=|Z|\leq\frac\delta{5a-\eta},
\]
because each original template row has degree $5a$. Delete $Z$ after
$B$. During this deletion its degrees are still at most $\eta$, while
$E\leq2E_0$ and $\kappa\geq\kappa_0/2$ for sufficiently small
$\delta_0$. By (85),
\[
 L_{z_t}(x)-\kappa(W,z_t)d_{z_t}(x)
 \leq-\frac{\kappa_0}{4}d_{z_t}(x)\qquad(x\in Z).
\]
If $D_Z$ denotes the edge mass removed in this step, then
\[
 D_Z=\int_0^1\int_Z d_{z_t}(x)\,dx\,dt.
\]
Thus deleting $Z$ increases $F$ by at least $c_ZD_Z$ for a constant
$c_Z>0$. The deletion of a measurable set is equivalently its
replacement by isolates.

\emph{Assign the remaining rows to clone types.} Let
$M=[0,1]\setminus(B\cup Z)$. Every point of $M$ is within $\eta$ of
a unique nonzero reference row. Call its type $\sigma(x)$; let
$\tau(x)$ be its original part label, and put
$M_i=\{x\in M:\sigma(x)=i\}$. Distinct Clebsch rows have distance
at least $6a$, so the measure $\nu$ of the misassigned set
$\{\sigma\ne\tau\}$ satisfies
\[
 \nu\leq\frac\delta{6a-\eta}.
\]
In particular, for $s_i=|M_i|$,
\begin{equation}\tag{89}\label{eq:reclassified-masses}
 |s_i-a|\leq\mu_B+\mu_Z+\nu=O(\delta).
\end{equation}
On $M$ let $T(x,y)=A_{\sigma(x),\sigma(y)}$ be the ideal Clebsch
step graphon. For almost every $x\in M$,
\begin{equation}\tag{90}\label{eq:reclassified-errors}
 \int_M|W(x,y)-T(x,y)|\,dy\leq\eta+\nu,\qquad
 e_M:=\lVert W|_{M^2}-T\rVert_1\leq\delta+2\nu.
\end{equation}
For the first inequality, the assigned reference row differs from its
reclassified version only on the misassigned set. For the second, $T$
differs from the original template on pairs having at least one
misassigned endpoint.

\emph{Repair all remaining edge errors.} Consider
$W_t=T+t(W|_{M^2}-T)$ for $0\leq t\leq1$, on the measure space $M$.
The row errors and global error remain bounded by the right-hand
sides of (90) throughout this path. By (89), the margin for $T$ is at
least $3g/4$. Equations (81), (87) and (90) imply
\[
 |(q_{W_t}-\kappa(W_t))-(q_T-\kappa(T))|
 \leq12\eta+O(\delta)<g/4.
\]
Hence $q_{W_t}-\kappa(W_t)\geq g/2$ where $T=1$, and it is at most
$-g/2$ where $T=0$. The variation $H=W|_{M^2}-T$ has the opposite
sign on each region. From (86),
\[
 \frac{d}{dt}F(W_t)\leq-c_Ee_M
\]
for a fixed $c_E>0$. Integrating gives
$F(T)-F(W|_{M^2})\geq c_Ee_M$.

\emph{Balance the classes and estimate the distance.} Put $q=|M|$
and $V=(\sum_i(s_i-q/16)^2)^{1/2}$. The Hessian estimate (78) applies
by (89). Combining the four comparisons gives
\begin{equation}\tag{91}\label{eq:four-defects}
 F(W_0)-F(W)\geq c_B\mu_B+c_ZD_Z+c_Ee_M+c_HV^2.
\end{equation}
Isolating $B$ costs at most $2\mu_B$ in $L^1$, and then isolating $Z$
costs exactly $2D_Z$. Replacing the remaining graphon by $T$ costs
$e_M$. Within $M$, balance the class masses by moving a set of measure
$\tfrac12\sum_i|s_i-q/16|$ between classes; atomlessness permits this
operation. The resulting graphon belongs to $\mathcal E_{\mathrm{Cl}}$,
and this last change costs at most $\sum_i|s_i-q/16|\leq4V$. Thus
\[
 d_1(W,\mathcal E_{\mathrm{Cl}})\leq2\mu_B+2D_Z+e_M+4V.
\]
Since $\mu_B,D_Z,e_M\leq1$, its square is at most
$16\mu_B+16D_Z+4e_M+64V^2$. Together with (91) this proves (76),
after multiplying by $5^{5/2}$ and adjusting the positive constant.

If $R_{\mathrm{ind}}(W)=R_\ast$, all four nonnegative terms in (91)
vanish. Thus $B$ is null, $Z$ consists of actual isolates, the
remaining graphon is exactly $T$, and its sixteen class masses are
equal. This proves the equality assertion without assuming that the
distance infimum is attained. The converse follows by homogeneity.
\end{proof}


\section*{Computational verification and reproducibility}

The computational proof consists of explicit finite rational identities.
The supplementary material provides the seven-vertex graph atlas, the
rooted flag spaces, primitive pattern counts, three rational witness
files, and their generators and verifiers. The main certificate contains
$184$ positive square terms, $37$ unrestricted identity columns, and
$826$ positive host remainders. The multipliers of $\Phi_\triangle$ and
$J_{\mathrm{del}}$ are positive, while those of $J_0$ and $J_1$ are zero.
Every one of the $1,044$ host identities is checked with exact integer or
rational arithmetic. A second implementation contracts integer Gram
matrices against independently recounted host tensors. The equality-span
witnesses are bound to the same main certificate and checked coordinate
by coordinate. Numerical optimization is used only to discover witnesses;
its floating-point output is not accepted as proof.

The supplementary programs also enumerate all $65,536$ Boolean Clebsch
profiles and all $6,496$ small-offset presentations under the convention
in Section~14.1. An independent enumeration of labeled exceptional
vertices gives the same canonical family. Formula (75) is recomputed for
every presentation, and the distinct polynomial formulas are additionally
checked on actual finite graphs. The branch counts in Section~14.1 have
a separate direct census check. These are verifications of the finite
inputs and construction formulas. The global bound, stability arguments,
and passage to eventual extremality are proved in the text. In
particular, the computations do not provide a numerical global threshold
$k_0(j)$.

The command \texttt{python3 src/verify\_project.py} reruns the complete
finite verification chain. Definitions, normalization conventions,
dependencies, and individual commands are specified in the accompanying
\texttt{README.md} and computational reproduction guide.

\paragraph{AI-assisted research methodology.}
ChatGPT (OpenAI) assisted with proof exploration, symbolic calculations,
implementation of generators and verifiers, and internal mathematical
cross-checks. Its mathematical output was treated as candidate reasoning,
not as a certificate of correctness. The supplied deterministic programs
perform the exact finite checks described above, and the analytic
arguments are given explicitly in the paper. 

\section*{Declaration of generative AI and AI-assisted technologies}

ChatGPT (OpenAI) was used in drafting and revising the manuscript and in
organizing computational documentation. Its use in mathematical
exploration and code development is described in the preceding
methodological statement. The author is responsible for the mathematical
claims, computations, references, and final text of this article.


\bibliographystyle{amsplain}

\begin{thebibliography}{10}

\bibitem{Balogh2016}
J{\'o}zsef Balogh, Ping Hu, Bernard Lidick{\'y}, and Florian Pfender,
  \emph{Maximum density of induced 5-cycle is achieved by an iterated blow-up
  of 5-cycle}, European Journal of Combinatorics \textbf{52} (2016), 47--58.

\bibitem{Bodnar2026Toolbox}
Levente Bodn{\'a}r, \emph{{FlagAlgebraToolbox}: Flag algebra computations in
  {SageMath}}, 2026, arXiv:2601.06590v2.

\bibitem{BodnarEtAl2026Six}
Levente Bodn{\'a}r, Jun Gao, Jared Le{\'o}n, Xizhi Liu, Oleg Pikhurko, and
  Shumin Sun, \emph{The inducibility of 6-vertex graphs}, 2026,
  arXiv:2606.00290v3.

\bibitem{BodnarPikhurko2026}
Levente Bodn{\'a}r and Oleg Pikhurko, \emph{Some exact inducibility-type
  results for graphs via flag algebras}, Electronic Journal of Combinatorics
  \textbf{33} (2026), no.~2, Paper No. P2.39.

\bibitem{Chao2025}
Ting-Wei Chao, Asaf {Cohen Antonir}, Anqi Li, and Hung-Hsun~Hans Yu, \emph{Edge
  inducibility via local directed graphs}, 2025, arXiv:2509.24064v1 [math.CO].

\bibitem{Davey2026}
Eoin Davey, Eoin Hurley, R{\'e}mi {de Joannis de Verclos}, Ross~J. Kang, and
  Jan Volec, \emph{Local flag algebras}, 2026, arXiv:2607.12461v1 [math.CO].

\bibitem{Lidicky2023}
Bernard Lidick{\'y}, Connor Mattes, and Florian Pfender, \emph{{$C_5$} is
  almost a fractalizer}, Journal of Graph Theory \textbf{104} (2023), no.~1,
  220--244.

\bibitem{LovaszSzegedy2006}
L{\'a}szl{\'o} Lov{\'a}sz and Bal{\'a}zs Szegedy, \emph{Limits of dense graph
  sequences}, Journal of Combinatorial Theory, Series B \textbf{96} (2006),
  no.~6, 933--957.

\bibitem{LovaszSzegedy2007}
\bysame, \emph{Szemer{\'e}di's lemma for the analyst}, Geometric and Functional
  Analysis \textbf{17} (2007), no.~1, 252--270.

\bibitem{MotzkinStraus1965}
T.~S. Motzkin and E.~G. Straus, \emph{Maxima for graphs and a new proof of a
  theorem of {Tur{\'a}n}}, Canadian Journal of Mathematics \textbf{17} (1965),
  533--540.

\bibitem{Pikhurko2019}
Oleg Pikhurko, Jakub Slia{\v c}an, and Konstantinos Tyros, \emph{Strong forms
  of stability from flag algebra calculations}, Journal of Combinatorial
  Theory, Series B \textbf{135} (2019), 129--178.

\bibitem{Razborov2007}
Alexander~A. Razborov, \emph{Flag algebras}, Journal of Symbolic Logic
  \textbf{72} (2007), no.~4, 1239--1282.

\bibitem{Wang2026}
Yichen Wang, Xiamiao Zhao, and Mei Lu, \emph{Edge version of the inducibility
  via the entropy method}, Electronic Journal of Combinatorics \textbf{33}
  (2026), no.~2, Paper No. P2.17, 16 pp.

\end{thebibliography}
\providecommand{\bysame}{\leavevmode\hbox to3em{\hrulefill}\thinspace}
\providecommand{\MR}{\relax\ifhmode\unskip\space\fi MR }
\providecommand{\MRhref}[2]{%
  \href{http://www.ams.org/mathscinet-getitem?mr=#1}{#2}
}
\providecommand{\href}[2]{#2}

\end{document}